\documentclass[11pt]{article}

\usepackage{mathtools}
\usepackage{geometry}
\usepackage{graphicx}
\usepackage[misc,geometry]{ifsym}
\usepackage{subfig}
\usepackage{tikz}
\usetikzlibrary{arrows}
\usepackage{url}
\usepackage{verbatim}
\usepackage{amsfonts}
\usepackage{caption}
\usepackage{tabularx}
\usepackage{array}
\usepackage{booktabs}
\usepackage{multirow}
\usepackage{fancyvrb}
\usepackage{amsmath}
\usepackage{comment}
\usepackage{algorithm}
\usepackage{algpseudocode}
\usepackage{enumerate}
\usepackage{amsthm}
\usepackage{latexsym}
\usepackage{amsmath,amssymb}
\graphicspath{{experimentsPaper/}}

\usepackage{blindtext,titlefoot}

\def\E{\mathbb{E}}

\def\Pr{{\mathbb{P}}}
\def\Prob{\Pr}

\newcommand{\given}{\middle|}

\newcommand{\ignore}[1]{}

\theoremstyle{plain}
\newtheorem{theorem}{Theorem}
\newtheorem{lemma}[theorem]{Lemma}
\newtheorem{corollary}[theorem]{Corollary}

\newtheorem{remark}[theorem]{Remark}

\newtheorem{definition}[theorem]{Definition}

\newcounter{rot}

\begin{document}

\thispagestyle{empty}

\title{Posterior Consistency for a Non-parametric Survival Model under a Gaussian Process Prior \thanks{ 
Tamara Fern\'andez was supported by funding from Becas CHILE.
}}
\author{
Tamara Fern\'andez
\and Yee Whye Teh
}
\date{\today}

\maketitle
\unmarkedfntext{Authors' address: Department of Statistics, University of Oxford, UK.}
\unmarkedfntext{Email: {\tt \{fernandez,y.w.teh\}@stats.ox.ac.uk}}

\thispagestyle{empty}

\begin{abstract}
In this paper, we prove almost surely consistency of a Survival Analysis model, which puts a Gaussian process, mapped to the unit interval, as a prior on the so-called hazard function. We assume our data is given by survival lifetimes $T$ belonging to $\mathbb{R}^{+}$, and covariates on $[0,1]^d$, where $d$ is an arbitrary dimension.
We define an appropriate metric for survival functions and prove posterior consistency with respect to this metric. Our proof is based on an extension of the theorem of Schwartz (1965), which gives general conditions for proving almost surely consistency in the setting of non 
i.i.d random variables. Due to the nature of our data, several results for Gaussian processes on $\mathbb{R}^+$ are proved which may be of independent interest.
\end{abstract}

\setcounter{page}{0}

\section{Introduction}
Consider a non-negative random variable $T$ representing the waiting time until some event of interest happens. Moreover, assume that associated to that time, there exits a covariate $X$ taking values on $[0,1]^d$. Given this set-up, the aim of Survival analysis is to estimate the underlying  survival function $S(t)=\Prob(T>t)$, and to asses the effect of the covariates $X$ on the distribution of these waiting times. In \cite[Fern\'andez et al. (2016)]{fernandez}, the authors propose a non-parametric Bayesian model for this type of data. The approach is based on defining a prior over the hazard function $\lambda_x=f_x/S_x$, where $f_x$ and $S_x$ stand for the probability density function and the survival function respectively. This prior is constructed by using a bounded and positive map $\sigma$, of a linear combination of Gaussian processes $(\eta_j(t))_{t\geq 0}$ and  covariates $X\in[0,1]^d$, i.e. 
\begin{eqnarray}
\lambda_x(t)&=&\sigma\left(\eta_0(t)+\sum_{j=1}^d x_j\eta_j(t)\right).\nonumber
\end{eqnarray}

The aim of this paper is verify almost surely consistency of the survival functions over a large class of possible stratifications of covariates. In order to do this, we assume the existence of a true parameter $\theta_0$, in some parameter space $\Theta$, under which our data was generated. Furthermore, we identify sufficient conditions over the prior $\Pi$ on $\Theta$, and true parameter $\theta_0\in\Theta$ which guarantees posterior consistency. 

Many have been the contributions from Bayesian non-parametric to the problem of survival analysis, as this problem has motivated the study of many different stochastic processes that attain good properties for modelling survival data. An important example are the so-called neutral-to-the-right random probability measures for modelling the cumulative density function, \cite[Doksum. (1977)]{doksum1974tailfree}. Within this class of processes are included the Dirichlet process and the Beta-Stacy process.  Other examples outside the latter class include the Beta process, which can be used as a prior for the cumulative hazard function, \cite[Hjort. (1990)]{hjort1990nonparametric} and the extended Gamma process  \cite[Dykstra and Laud. (1981)]{dykstra1981bayesian}, which defines a prior over the hazard rate function. See \cite[pages 87--99]{bayes2010} for a complete review of these methods. 	

On the other hand, there exists an extensive literature on posterior consistency for non-parametric Bayesian methods. Indeed there are some general results providing conditions under we achieve posterior consistency on infinite-dimensional spaces. Some of them are \cite[Doob, J. L. (1949)]{Doob1949}, \cite[Schwartz, L. (1965)]{schwartz1965bayes}, \cite[Barron, Schervish and Wasserman. (1999)]{barron1999consistency}. Moreover, models under Gaussian Process priors have also been subject of intensive study. Some relevant works include: \cite[Choi and Schervish. (2004)]{choi2004posterior} in the setting of regression,  \cite[Ghosal and Roy. (2006)]{ghosal2006posterior} in the setting of logistic regression and 
 \cite[Tokdar and Ghosh. (2007)]{tokdar2007posterior} in the setting of density estimation. 

The work of \cite[Tokdar and Ghosh. (2007)]{tokdar2007posterior} is of particular interest for us since it relates more closely with our setting. Nevertheless it differs from our approach not only because the model is different but, more importantly, because it considers distributions supported on bounded intervals. Indeed, one of the main challenges of our work is given by the nature of the data, the survival lifetimes belong to $\mathbb{R}^{+}$. Additionally, as we consider covariates, we add an extra difficulty to our work. For the latter, we adopt the approach given in \cite[Ghosal and Roy. (2006)]{ghosal2006posterior}, and consider a random and fixed design for them.

Finally, we use the extension of  Schwartz's theorem given in \cite[Choi and Schervish. (2004)]{choi2004posterior}. Then, our result is based on the existence of tests with exponentially small type I and type II error, and the positivity of Kullback-Leibler neighbourhoods of the true parameter. In order to verify these conditions, we use results from the Vapnik-Chervonenkis (VC) theory for the tests, and results from RKHS of Gaussian process priors for the Kullback-Leibler neighbourhoods. Good reviews for both theories can be found on \cite[Devroye, Luc and Lugosi. (2012)]{devroye2012combinatorial} and \cite[Van der Vaart and van Zanten. (2008)]{van2008reproducing} respectively.

The paper is organized as it follows. In section 2, we review the model introduced in \cite[Fern\'andez et al.]{fernandez}. In section 3 we give a brief review of Schwartz Theorem for i.i.d. random variables and provide the extended version for independent but non- identically distributed given by \cite[Choi and Schervish. (2004)]{choi2004posterior}. In section 4 we define the metric used for survival functions associated to a set of covariates and moreover, we describe the assumptions made over the design of covariates, prior and true parameter for achieving posterior consistency. The assumptions made in these sections are use to prove the main theorems in section 5. Finally in section 6 we provide a brief discussion.
\section{Model}\label{SectionModel}
The model in \cite[Fern\'andez et al. (2016)]{fernandez} is given by the following hierarchical structure
\begin{eqnarray}\label{eqn:model}
T_i|\lambda,X_i=x &\overset{ind}{\sim}&f_x(t),\nonumber\\
f_x(t)&=&\lambda_x(t)e^{-\int_{0}^{t}\lambda_x(s)ds},\nonumber\\
\lambda_x(t)&=&\Omega\sigma\left(\eta_0(t)+\sum_{j=1}^dx_j\eta_j(t)\right),\\
\Omega&\sim&\nu,\nonumber\\
\eta_j(\cdot)&\overset{ind.}{\sim}& \mathcal{GP}(0,\kappa_j(\cdot)),j=0,\ldots,d, \text{ independently of }\Omega,\nonumber
\end{eqnarray}
in which $\sigma(x)=(1+e^{-x})^{-1}$, and $\kappa_j$ are stationary covariance functions. In particular, the boundedness of the link function $\sigma$ is a necessary condition for applying the inference scheme used in \cite[Fern\'andez et al. (2016)]{fernandez}, but not for giving a correct definition of a prior over hazard functions. For the latter we could arbitrarily use any non-negative map.

For simplicity, we adopt this particular assumption, but at the end of the paper discuss its relevance for our results.

A nice property of this model is that $\E(\lambda_x(t)|\Omega)=\Omega/2$, which corresponds to the hazard function of an exponential random variable with mean $2/\Omega$. This implies we can construct distributions over hazard functions centred on parametric models. Additionally, we remark that the above model has been proved to be well-defined in \cite[Fern\'andez et al. (2016)]{fernandez}, in the sense that under the assumption
\begin{eqnarray}
\int_0^T\kappa_j(t)dt&=&o(T),\label{kernel o(t)}
\end{eqnarray}
for all $j\in\{0,\ldots,d\}$, the conditional distribution function of the survival lifetimes $T$, given a covariate $X=x$, is a proper distribution function with probability 1, i.e.
\begin{eqnarray}
\Pi\left(\lim_{t\to\infty}S_{X}(t,\theta)=1\given X=x\right)&=&1.\nonumber
\end{eqnarray}
Notice that condition \ref{kernel o(t)} holds for all the stationary kernels which decreases to zero.
\section{Posterior Consistency}
Consider the parameter space $\Theta=\mathbb{R}^{+}\times\mathcal{F}^{d+1}$, where $\mathcal{F}$ denotes a space of functions. We assume that there exists a true parameter $\theta_0=(\Omega_0,\eta_{0,0},\ldots,\eta_{d,0})$ on $\Theta$, which determines the true hazard function $\lambda_{x,0}(t)=\Omega_0\sigma\left(\eta_{0,0}(t)+\sum_{j=1}^dx_j\eta_{j,0}(t)\right)$ defined in model \eqref{eqn:model}. The aim of this paper is to find conditions over the prior $\Pi$ on $\Theta$ and true parameter $\theta_0\in\Theta$, to ensure this model achieves posterior consistency of the survival functions in $\mathbb{R}^{+}$. 

Schwartz, L. (1965)\cite{schwartz1965bayes} gave conditions for proving posterior consistency at $\theta_0$ under a loss function $d(\theta,\theta_0)$ in the context of independent and identically distributed random variables. These conditions include the existence of a sequence of tests with exponentially small type I and type II error, and prior positivity of Kullback-Leibler neighbourhoods of the true parameter $\theta_0$. A bit more involved scenario is when we do not have identically distributed observations. Choudhuri, Ghosal and Roy (2004) \cite{choudhuri2004bayesian}  extended Schwartz's Theorem for the case of non i.d.d. observations for the case of convergence in-probability.  Later, Choi and Schervish (2004)\cite{choi2004posterior} provided an alternative extension for almost surely convergence.

\begin{theorem}[Choi and Schervish \cite{choi2004posterior}]\label{TheoChoi}
Let $(Z_i)_{i=1}^\infty$ be independently distributed with densities $(f_i(\cdot,\theta))_{i=1}^\infty$, with respect to a common $\sigma$-finite measure, where the parameter $\theta$ belongs to an abstract measurable space $\Theta$. The densities $f_i(\cdot;\theta)$ are assumed to be jointly measurable. Let $\theta_0\in\Theta$ and let $P_{\theta_0}$ stand for the joint distribution of $(Z_i)_{i=1}^\infty$ when $\theta_0$ is the true value of $\theta$. Define the loss function $d(\theta,\theta_0)$ and let $U_\epsilon=\left\{\theta:d(\theta,\theta_0)\leq\epsilon\right\}$ be a set of $\Theta$ for all $\epsilon>0$. Let $\theta$ have a prior $\Pi$ on $\Theta$. Define
\begin{eqnarray}
\Upsilon(\theta_0,\theta)&=&\log\frac{f_i(Z_i;\theta_0)}{f_i(Z_i;\theta)},\nonumber\\
K_i(\theta_0,\theta)&=&\E_{\theta_0}(\Upsilon(\theta_0,\theta)),\nonumber\\
V_i(\theta_0,\theta)&=&\mathbb{V}ar_{\theta_0}(\Upsilon(\theta_0,\theta)).\nonumber
\end{eqnarray}  
\begin{itemize}
\item[(T1)]Prior positivity of neighbourhoods.\\
Suppose that there exits a set $B$ with $\Pi(B)>0$ such that 
\begin{itemize}
\item[(i)]$\sum_{i=1}^\infty\frac{V_i(\theta_0,\theta)}{i^2}<\infty, \forall\theta\in B$,
\item[(ii)] For all $\varepsilon>0$, $\Pi(B\cap\{\theta:K_i(\theta_0,\theta)<\varepsilon\text{ for all } i\})>0$
\end{itemize}
\item[(T2)]Existence of tests.\\
Suppose that there exit test functions $(\Phi_n)_{n=1}^\infty$ and constants $C_1,c_1>0$ such that
\begin{itemize}
\item[(i)]$\sum_{n=1}^\infty\E_{\theta_0}\Phi_n<\infty$,
\item[(ii)]$\underset{\theta\in U_\epsilon^c}{\sup}\E_{\theta}(1-\Phi_n)\leq C_1e^{-c_1 n}$,
\end{itemize}
Then for all $\epsilon>0$
$$\Pi(\theta\in U_\epsilon^c|Z_1,\ldots,Z_n)\to 0\hspace{1cm}a.s.[\Prob_{\theta_0}].$$
\end{itemize}
\end{theorem}

\begin{remark}
It is worth noticing that for i.i.d. samples, we recover Schwartz's theorem by replacing condition (T1) with the following condition.
\begin{itemize}
\item[(G1)] The parameter $\theta_0\in\Theta$ is such that for every $\varepsilon>0$,
\begin{eqnarray}
\Pi\left(\theta:\E_{\theta_0}\left(\log\frac{f(Z,\theta_0)}{f(Z,\theta)}<\varepsilon\right)\right)&>&0.\nonumber
\end{eqnarray}
\end{itemize}
in which $f(Z,\theta)$ does not depend on $i$, as we have i.i.d samples. Moreover, we do not need any condition on the variance.
\end{remark}
\section{Framework}
In this section we introduce all the necessary elements prior introducing the main result.
\subsection{Notation}
In order to introduce the notation, we first specify some general assumptions on how the covariates arise in the model.
\begin{itemize}
\item \textbf{[RD] Random design:} We assume the covariates $X$ are distributed according a probability distribution $Q$ on $[0,1]^d$.
\item \textbf{[NRD] Non random design:} In this case, we assume the covariates are fixed. In particular, given the sequence of covariates $(x_{i})_{i\geq 1}$, we state our results as function of their empirical distribution $Q_n=n^{-1}\sum_{i=1}^n\delta_{x_{i}}$.  
\end{itemize}

For the random design, notice that conditioned on $X=x$, the probability density function for the survival lifetimes times is fully specified by the parameter $\theta=(\Omega,\eta_0,\eta_1,\ldots,\eta_d)\in\Theta$, then we use $f_{x}(t,\theta)$ when referring  to it. We use the same notation for cumulative density function $F_x(t,\theta)$, the hazard function $\lambda_{x}(t,\theta)$ and the cumulative hazard function $\Lambda_{x}(t,\theta)$ given the covariate $X=x$. We adopt the same notation for the fixed design case, but understand from the context that $f_x(t,\theta)$ makes reference to the distribution of times under the fixed covariate $x$, i.e $x$ is not random, nor does $f_x(t,\theta)$ denote a conditional density function.

In the same way, for the random design case, we denote the join probability measure on pairs $\{(T_i,X_i)\}_{i\geq 1}$ under the parameter $\theta$, by $\Prob_{\theta}$. We also denote by $\mu_{\theta}(A) = \Prob_{\theta}((T_i,X_i) \in A)$, for any $i$ (remember the pairs $(T_i, X_i)$ are i.i.d), and $A\subseteq \mathbb{R}^{+}\times[0,1]^d$. For the fixed design case, our probability measure depends on $\theta$ as well as on the covariates $\boldsymbol{x} = (x_i)_{i \geq 1}$. We denote by $\Prob_{\theta}(\cdot | \boldsymbol{x})$ the joint distribution of $(T_i)_{i \geq 1}$ associate with the fix sequence of covariates $(x_i)_{i \geq 1}$. Sometimes, to short notation we write $\tilde \Prob_{\theta}(\cdot) = \Prob_{\theta}(\cdot| \boldsymbol{x})$. We define $\tilde \mu_{\theta}^{x_i}([a,b]) = \Prob_{\theta}(T_i \in [a,b]| \boldsymbol{x})$. Note that all measures $\tilde{\mu}_{\theta}^{x_i}$ are potentially different since the distribution of $T_i$ depends on its covariate $x_i$.

\subsection{Metric for Survival Analysis}
We provide a different metric for each design of covariates, \textbf{[RD]} and \textbf{[NRD]}. As we already stated, our object of interest are the survival functions $S_x(t)$ for $x\in [0,1]^d$. Under the further assumption that the covariates $X$ follow a distribution $Q$ on $[0,1]^d$ (\textbf{[RD]}), a sensible way of assessing the quality of our estimators is by considering
\begin{eqnarray}
d_B(\theta,\theta_0)&=&\underset{t\in\mathbb{R}^{+}}{\sup}\left|\int_{B}(S_x(t,\theta)-S_x(t,\theta_0))Q(dx)\right|,
\end{eqnarray}   
with $B\subseteq[0,1]^d$. This election is related with the idea of stratifying the space of covariates and measuring how good is the estimator in the different strata under the distribution $Q$. Hence, it will be important to check the value of $d_B(\theta,\theta_0)$ over a large class of subsets $B$ on $[0,1]^d$. Let $\mathcal{A}$ be the class of rectangles in $\mathbb{R}^{+}\times [0,1]^d$, then we define the loss function,
\begin{eqnarray}
d(\theta,\theta_0)&=&\underset{A\in\mathcal{A}}{\sup}|\mu_{\theta}(A)-\mu_{\theta_0}(A)|.\label{eqn: loss function}
\end{eqnarray}

On the other hand, under the assumption of fixed covariates (\textbf{[NRD]}), we replace the distribution $Q$ by the empirical distribution $Q_n=\sum_{i=1}^n\delta_{x_i}$. In this way, for a set $A=A_1\times A_2\subseteq\mathbb{R}^{+}\times[0,1]^d$, the loss function becomes
\begin{eqnarray}
d_n(\theta,\theta_0)&=&\underset{A\in\mathcal{A}}{\sup}\left|\int_{A_2}(\tilde{\mu}_{\theta}^x(A_1)-\tilde{\mu}_{\theta_0}^x(A_1))Q_n(dx)\right|\nonumber\\
&=&\underset{A\in\mathcal{A}}{\sup}\left|\sum_{i=1}^n\frac{\delta_{x_{i}}(A_2)}{n}(\tilde{\mu}_{\theta}^{x_i}(A_1)-\tilde{\mu}_{\theta_0}^{x_i}(A_1))\right|\label{abuse notation}
\end{eqnarray}

\subsection{Non-stationary Gaussian Processes}
The main difficulty of this work is that the index space $\mathcal{T}$, of the Gaussian processes $(\eta_j(t))_{t\in \mathcal{T}}$ defined in equation  \eqref{eqn:model},  is $\mathbb{R}^{+}$. Up to our knowledge, most of the results proving posterior consistency under a Gaussian process prior use the fact that these processes are defined in a $L^{\infty}$ space of functions, condition which is not satisfied in our case. A typical choice is to put a Gaussian prior on the space of continuous functions on a closed interval, a seen \cite{choi2004posterior}, \cite{ghosal2006posterior} and \cite{tokdar2007posterior}. The reason for this, is that the results need to compute small ball probabilities. These have been vastly studied in the setting of Gaussian measures on $L^{\infty}$, providing good tools for finding appropriate bounds for these probabilities. 

To bypass the problem that our processes are not in  $L^{\infty}(\mathbb{R}^{+})$, we consider a simple modification of the original processes, belonging to this space and thus allowing us to use the machinery of RKHS to deal with small ball probabilities. After that, we translate the results of the modified processes to the original ones.

Define the positive and continuous piecewise function,
\begin{eqnarray}
h_d(t)&=&\left\{ \begin{array}{cc}
\frac{d+1}{1+\log(1-e^{-1})}&t\leq 1\\
\frac{d+1}{t+\log(1-e^{-t})}&t>1
\end{array}\right.\label{eq: h function}
\end{eqnarray}
where $d$ stands for the dimension of the covariate $X$. We define the non-stationary Gaussian processes $(\hat{\eta}_j(t))_{t\geq 0}$, as function of the original Gaussian processes $(\eta_j(t))_{t\geq 0}$ defined in equation \eqref{eqn:model}. Indeed, we define, 
\begin{eqnarray}
\hat{\eta}_j(t)&=&h_d(t)\eta_j(t)\label{eq: non-stationary GP}
\end{eqnarray}
for all $t\geq 0$ and all $j\in\{0,\ldots,d\}$. We prove that the new Gaussian processes 
$(\hat{\eta}_j(t))_{t\geq 0}$ lie in the Banach space of continuous functions tending to zero, endowed with the uniform norm. Moreover, we proved the above space is separable.

\begin{lemma}\label{Lemma: separable}
Let $\mathcal{C}_0$ denote the space of all the continuous functions from $\mathbb{R}^{+}$  to $\mathbb{R}$ tending 0 as $t$ goes to infinity. Then the Banach space $\mathbb{B}=(\mathcal{C}_0,\left\|\right\|_{\infty})$ of functions in $\mathcal{C}_0$ endowed with the uniform norm is separable.
\end{lemma}
\begin{lemma}\label{Lemma: bounded supremum}
The set of sample paths of $\hat\eta_i$ is subset of the separable Banach space $\mathbb{B}=(\mathcal{C}_0,\left\|\right\|_{\infty})$ for every $i\in\{0,\ldots,d\}$.
\end{lemma}

\subsection{Assumptions}\label{section: assumptions}
In this section, we make some assumptions over the prior $\Pi$ and the parameter space $\Theta$, which allow us to prove posterior consistency. Assumptions (A1), (A2) and (A4) will be common to both covariates designs, 
while the third assumption will be  design-specific. Indeed, (A3) will be used for the random design case (\textbf{[RD]}), and (A3') will be used for the fixed design case (\textbf{[NRD]}).

\begin{itemize}
\item[(A1)] The stationary covariance function $\kappa_j(t)$ is such that $(\kappa_j(0)-\kappa_j(2^{-n}))^{-1}\geq n^6$ for all $j\in\{0,\ldots,d\}$.
\item[(A2)] In addition $\nu$ assigns positive probability to every neighbourhood of the true parameter $\Omega_0$, i.e. for every $\delta>0$, $\nu\left(\left|\frac{\Omega}{\Omega_0}-1\right|<\delta\right)>0$.
\item[(A3)]Finite first moment under the distribution of the true parameter, i.e. 
\begin{eqnarray}
\E_{\theta_0}(T)=\E(\E_{\theta_0}(T)|X=x))
=\int_{[0,1]^d}\int_{0}^{\infty}tf_x(t,\theta_0)dtQ(dx)<M<\infty\nonumber
\end{eqnarray}
\item[(A3')]We assume that given $\delta>0$, there exits $M\in[0,\infty)$ such that 
\begin{eqnarray}
\E_{\theta_0}(T^2\textbf{1}{\{T>M\}}|x)&=&\int_M^\infty t^2f_x(t,\theta_0)dt<\delta\nonumber
\end{eqnarray}
for all $x\in[0,1]^d$, i.e. the class of random variables $T^2$ indexed in the covariate $x\in[0,1]^d$ is uniformly integrable under the true parameter $\theta_0$. 

\item[(A4)]The true parameters $\eta_{j,0}$ take the form $\hat{\eta}_{i,0}/h_d$, where $\hat{\eta}_{j,0}$ is in the support of $\hat\eta_i$ under the uniform norm, for all $j\in\{1,\ldots,d\}$.
\end{itemize}

\section{Main Result}
\textbf{Random design case}

We first state the posterior consistency results under the loss function given by equation \eqref{eqn: loss function}, for the case where the covariates are drawn from the probability distribution $Q$. Observe that there is not an unique interpretation for this result under this assumption. For instance, a first approach is to consider the covariates $X$ as data generated from our model under the distribution $Q$. In this way, the first posterior consistency result we can think of, is in terms of the joint probability distribution $\Prob_{\theta_0}$ of the pairs $((T_i,X_i))_{i\geq 1}$. This is stated in the following theorem.\\

\begin{theorem}[]\label{THMain 1} Let $((T_i,X_i))_{i=1}^{\infty}$ be a sequence of independent random variables, where $T_i\in\mathbb{R}^{+}$ denotes a random survival time, and $X_i\in[0,1]^d$ its associated covariate. Suppose that the covariates arise according a random design with probability distribution $Q$ on $[0,1]^d$, and that conditioned on $X_i$, the random failure times $T_i$ are generated according the model described in equation \eqref{eqn:model}. Let $\Prob_{\theta_0}$ denote the joint probability measure of $((T_i,X_i))_{i=1}^{\infty}$ under the true parameter $\theta_0$. Assume that assumptions (A1), (A2), (A3) and (A4) are satisfied, then for every $\epsilon>0$,
\begin{eqnarray}
\lim_{n\to\infty}\Pi\left(d(\theta,\theta_0)>\epsilon|(T_1,X_1),\ldots,(T_n,X_n)\right)=0,\hspace{0.5cm}\Prob_{\theta_0}\hspace{0.1cm}a.s,\nonumber
\end{eqnarray}
with $d(\theta,\theta_0)$ defined in equation \eqref{eqn: loss function}.
\end{theorem}
Another approach is by taking covariates that effectively follow a distribution $Q$ but are not considered to be generated by the model. To make an explicit distinction between this case and the latter, we denote the posterior distribution by a different letter $\Pi_2$. Notice that the likelihood under this approach is different than the one in the first case, as we consider the conditional distribution of the times $T$ given the covariates $X$, but under the reasonable assumption that the distribution $Q$ is unrelated to the parameter $\theta$, we recover exactly the same posterior distribution. A direct consequence of this, is that Theorem \ref{THMain 1} can be used to deduce the following corollary. 

\begin{corollary}
By the definition of conditional expectation, the latter result can be interpreted as that for every $\epsilon>0$,
\begin{eqnarray}
\E\left(\Prob_{\theta_0}\left(\lim_{n\to\infty}\Pi_2\left(d(\theta,\theta_0)>\epsilon|(T_1,X_1),\ldots,(T_n,X_n)\right)=0\given (X_n)_{n\geq 1}\right)\right)=1,\nonumber
\end{eqnarray}
which implies 
\begin{eqnarray}
Q\left(\Prob_{\theta_0}\left(\lim_{n\to\infty}\Pi_2\left(d(\theta,\theta_0)>\epsilon|(T_1,X_1),\ldots,(T_n,X_n)\right)=0\given (X_n)_{n\geq 1}\right)=1\right)=1,\nonumber
\end{eqnarray}
i.e. $Q$ almost surely convergence of the joint conditional distribution of the times $(T_i)_{i\geq 1}$ given the covariates $(X_i)_{i\geq 1}$ under the assumption of the true parameter $\theta_0$. 
\end{corollary}

The last approach is to consider a random design in which the distribution $Q$ is an unknown parameter. Following the Bayesian approach, we specify a prior over $Q$ and assume that the distribution $Q$ is unrelated to $\theta$. It follows that the likelihood for $\theta$ can be separated out from that of $Q$. Then, if $\theta$ and $Q$ are independent a priori, so they are a posteriori. Assuming a true parameter $Q_0$, the independence of $\theta$ and $Q$ allows us to exactly replicate the proof of Theorem \ref{THMain 1} for this case, when assumption (A3) holds for $Q_0$. Notice that in this case we assume the covariates are generated by $Q_0$.

\textbf{Non-random design case}

We consider the version of the theorem for the non-random design approach.
\begin{theorem}[Consistency on the Real line]\label{THMain 2} Let $(T_i)_{i\geq 1}$ be a sequence of independent random variables on $\mathbb{R}^{+}$ denoting survival lifetimes. Additionally consider the sequence of associated fixed covariates $(x_{i})_{i\geq 1}$. Suppose that the distribution of  random survival lifetimes $T$ for a fixed value of a covariate $x$ is generated according the model described in equation \eqref{eqn:model}. Let $\tilde{\Prob}_{\theta_0}$ define the joint probability measure of $(T_i)_{i\geq 1}$ under the true parameter $\theta_0$. Assume that assumptions (A1), (A2), (A3') and (A4) are satisfied, then for every $\epsilon>0$,
\begin{eqnarray}
\lim_{n\to\infty}\Pi\left(d_n(\theta,\theta_0)>\epsilon|(T_1,x_1),\ldots,(T_n,x_n)\right)=0,\hspace{0.5cm}\tilde{\Prob}_{\theta_0}\hspace{0.1cm}a.s,\nonumber
\end{eqnarray}
with $d_n(\theta,\theta_0)$ defined in equation \eqref{abuse notation}. 
\end{theorem}
\subsection*{Overview of proofs}
We give an overview of the proofs and lemmas used to prove Theorem  \ref{THMain 1} and Theorem \ref{THMain 2}. As many of these results will be used by both proofs, we classify them into two general subsections related to conditions (T1) and (T2) of Theorem \ref{TheoChoi}: Existence of tests and Prior Positivity of neighbourhoods.

It is worth noticing that for the random design approach (\textbf{[RD]}), the pairs $((T_i,X_i))_{i\geq 1}$ are i.i.d. samples from the joint distribution of time and covariates and hence Schwartz's Theorem is enough for proving Theorem \ref{THMain 1}. In particular, for this design, we exchange condition (T1) of Theorem \ref{TheoChoi}  by the alternative condition (G1). For the non-random design (\textbf{[NRD]}), we need the version given in Theorem \ref{TheoChoi}.   
\subsubsection*{Existence of Tests}
Given the metrics in equations \eqref{eqn: loss function} and \eqref{abuse notation}, a natural choice of a test is to consider the largest deviation of the joint empirical distribution $\mu_n$ with respect the true joint distribution $\Prob_{\theta_0}$ over the class of rectangles $\mathcal{A}$ on $\mathbb{R}^{+}\times [0,1]^d$. More formally, for $A=A_1\times A_2\in\mathcal{A}$, we define the empirical distribution as
\begin{eqnarray}
\mu_n(A)&=&\frac{1}{n}\sum_{i=1}^n\delta_A(T_i,X_i),
\end{eqnarray}
and then, we define the sequence of tests 
\begin{eqnarray}
\phi_n((T_1,X_1),\ldots,(T_n,X_n))&\coloneqq&\mathbf{1}\left\{\underset{A\in\mathcal{A}}{\sup}\left|\mu_n(A)-\mu_{\theta_0}(A)\right|>\frac{\epsilon}{4}
\right\},\end{eqnarray}
for the random design case associated with the loss function given in equation \eqref{eqn: loss function}, and 
\begin{eqnarray}
\tilde{\phi}_n((T_1,X_1),\ldots,(T_n,X_n))&\coloneqq &\mathbf{1}\left\{\underset{A\in\mathcal{A}}{\sup}\left|\mu_n(A)-\int_{A_2}\tilde{\mu}_{\theta_0}^x(A_1)Q_n(dx)\right|>\frac{\epsilon}{4}
\right\},\end{eqnarray}
for the fixed design case associated with the loss function given in equation \eqref{abuse notation}.

We use results from the Vapnik-Chervonenkis (VC) theory for providing exponentially small bounds for the sequence of probabilities $\E_{\theta_0}(\phi_n)$ and $\tilde{\E}_{\theta_0}(\tilde{\phi}_n)$. These bounds will require the finiteness of the VC dimension $V$, of the class of sets $\mathcal{A}$ considered. We recall the definition of the VC dimension $V$ as the largest integer $n$ such that 
\begin{eqnarray}
S_{\mathcal{A}}(n)&=&2^n,\label{VC}
\end{eqnarray}
where $S_{\mathcal{A}}(n)=\underset{x_1,\ldots,x_n\in\mathbb{R}^d}{\max}|\{\{x_1,\ldots,x_n\}\cap A;A\in\mathcal{A}\}|$, i.e. the maximal number of different subsets of n points which can be separated by elements of $\mathcal{A}$. A direct consequence of this requirement, is that we can modify our result and include a larger class of sets $\mathcal{A}$, as long as this new class has a finite VC dimension. For our particular choice, rectangles in $\mathbb{R}^{+}\times [0,1]^d$, $V$ is bounded by $2(d+1)$, \cite[Lemma 4.1]{devroye2012combinatorial}. On the other hand, if we consider the class of all measurable sets on $\mathbb{R}^{+}\times [0,1]^d$, the VC dimension is then infinity and thus the result we use can not be applied.

\begin{lemma} \label{Lemma: Testing Conditions without censoring}
Let $\epsilon>0$ and $\Prob_{\theta_0}$ be the true joint probability measure defined by the true parameter $\theta_0$, then
\begin{eqnarray}
\E_{\theta_0}(\phi_n)\leq Ce^{-\frac{\epsilon^2}{2}n}&,&\underset{\theta\in\Theta,d(\theta_0,\theta)>\epsilon}{\sup}\E_{\theta}(1-\phi_n)\leq Ce^{-\frac{\epsilon^2}{2}n},\nonumber
\end{eqnarray}
for some constant $C$.
\end{lemma}
\begin{lemma} \label{Lemma: Testing Conditions without censoring Qn}
Let $\epsilon>0$ and $\tilde{\Prob}_{\theta_0}$ be the true joint probability measure defined by the true parameter $\theta_0$, then
\begin{eqnarray}
\tilde{\E}_{\theta_0}(\tilde{\phi}_n)\leq \tilde{C}e^{-\frac{\epsilon^2}{2}n}&,&\underset{\theta\in\Theta,d_n(\theta_0,\theta)>\epsilon}{\sup}\tilde{\E}_{\theta}(1-\tilde{\phi}_n)\leq\tilde{C}e^{-\frac{\epsilon^2}{2}n},\nonumber
\end{eqnarray}
for some constant $\tilde{C}$.
\end{lemma}
\subsubsection*{Prior Positivity of neighbourhoods}
We proceed to prove the conditions related to Kullback-Leibler neighbourhoods of the true parameter $\theta_0$. In order to do this, we define a set $B_{\delta,\tau}$, in which conditions (G1) and (T1) will be satisfied for both settings. We finish our result by proving the positivity of such set. 
\begin{definition}
Let $\delta>0$ and $\tau>1$, then
\begin{eqnarray}
B_{\delta,\tau}&\coloneqq&\left\{\theta\in\Theta:\left|\frac{\Omega}{\Omega_0}-1\right|<\delta;\underset{t\in[0,\tau]}{\sup}|\eta_j(t)-\eta_{j0}(t)|\leq\frac{\delta}{1+\tau}, \underset{t>\tau}{\inf}\{\eta_j(t)h_d(t)\}>-1,\forall j\right\}.\label{eqn: setB}
\end{eqnarray}
\end{definition} 
Then, we show that for a given $\epsilon>0$, we can always find some $\delta(\epsilon)>0$ and $\tau(\epsilon)>1$ such that $KL(\theta_0,\theta)<\epsilon$ for all $\theta\in B_{\delta(\epsilon),\tau}$ for all $\tau\geq \tau(\epsilon)$, in the case of condition (G1). The intuition behind this proof is that since the KL is finite (and defined as an integral), from some point onwards, we do not care about the tail of the Gaussian process. We quantify the notion of  ``not caring'' by multiplying the Gaussian process by the decreasing function $h_d$ from some time $\tau(\epsilon)$.  The same applies for condition (T1). We formalize this result in the two following lemmas.

\begin{lemma}\label{Lemma: KL1}
Assume  the random design \textbf{[RD]} for the covariates and assumption (A3). Let $0<\epsilon<1$, then there exits $0<\delta(\epsilon)<\frac{1}{2}$ small enough and $\tau(\epsilon)>1$ large enough such that for all $\tau\geq\tau(\epsilon)$ and for all $\theta\in B_{\delta(\epsilon),\tau}$,
\begin{eqnarray}
\E_{\theta_0}\left(\log\frac{f_X(T,\theta_0)}{f_X(T,\theta)}\right)<\epsilon.\nonumber
\end{eqnarray}
\end{lemma}

\begin{lemma}\label{Lemma: KL2}
Assume the fixed design \textbf{[NRD]} for the covariates and assumption (A3'). Let $0<\epsilon<1$, then there exits $0<\delta(\epsilon)<\frac{1}{2}$ small enough and $\tau(\epsilon)>1$ large enough such that for all $\tau\geq\tau(\epsilon)$ and for all $\theta\in B_{\delta(\epsilon),\tau}$,
\begin{itemize}
\item[(i)] $K_i(\theta_0,\theta)<\epsilon$ for all $i$
\item[(ii)]$\sum_{i=1}^\infty\frac{V_i(\theta_0,\theta)}{i^2}<\infty$.
\end{itemize}
\end{lemma}

\subsubsection*{Positivity of $B_{\delta,\tau}$}
Lastly we verify, the prior positivity of the set $B_{\delta,\tau}$. Since this set is common to both, Lemma \ref{Lemma: KL1} and Lemma \ref{Lemma: KL2}, we just give a general proof and use it in the two settings. 

Indeed, for given $\delta(\epsilon)>0$ (as defined in Lemma \ref{Lemma: KL1} or \ref{Lemma: KL2}), we prove that there always exists $\tau^\star>1$ such that $B_{\delta(\epsilon),\tau}$ has positive measure for all $\tau\geq\tau^\star$. In this way, by finally taking $\tau\geq\max\{\tau(\epsilon),\tau^\star\}$, we are able to pick a set $B_{\delta(\epsilon),\tau}$, with positive prior probability and satisfying lemmas \ref{Lemma: KL1} or \ref{Lemma: KL2} (depending in which setting we are).

By independence of the Gaussian processes and parameter $\Omega$,
\begin{eqnarray}
\Pi(B_{\delta,\tau})=\prod_{j=0}^d\Pi\left(\underset{t\in[0,\tau]}{\sup}|\eta_{j}(t)-\eta_{j,0}(t)|\leq\frac{\delta}{1+\tau},\underset{t>\tau}{\inf}\{\eta_j(t)h_d(t)\}>-1\right)\nu\left(\left|\frac{\Omega}{\Omega_0}-1\right|<\delta\right)
\end{eqnarray}
with $\nu\left(\left|\frac{\Omega}{\Omega_0}-1\right|<\delta\right)>0$ by assumption (A3). Hence, we need to check the positivity of the sets
\begin{eqnarray}
\mathcal{G}^j_{\delta,\tau}&=&\left\{\underset{t\in[0,\tau]}{\sup}|\eta_j(t)-\eta_{j,0}(t)|\leq\frac{\delta}{1+\tau},\underset{t>\tau}{\inf}\{\eta_j(t)h_d(t)\}>-1\right\},
\end{eqnarray}
for every $j\in\{0,\ldots,d\}$.

In order to show this, we consider alternative events $\tilde{\mathcal{G}}^j_{\delta,\tau}$, in terms of the new processes $(\tilde{\eta}_j(t))_{t\geq 0}$ defined in equation \eqref{eq: non-stationary GP}, and show that from some $\tau_{S_j}>1$, 
\begin{eqnarray}
\tilde{\mathcal{G}}^j_{\delta,\tau}&\subseteq& \mathcal{G}^j_{\delta,\tau}, \forall \tau>\tau_{S_j}.\nonumber
\end{eqnarray} 

By Lemma \ref{Lemma: bounded supremum}, the processes $(\hat{\eta}_j(t))_{t\geq 0}$ can be seen as probability measures on the separable Banach space $\mathbb{B}=(\mathcal{C}_0,\|\|_{\infty})$. Then, by definition, for a function $\hat\eta_{j,0}$ in the support of  $\hat\eta_j$, there exits $\tau_{S_j}>1$ such that
\begin{eqnarray}
\hat\eta_{j,0}(t)&\geq&-\frac{1}{2} \hspace{0.5cm}\forall t\geq\tau_{S_j}.\nonumber
\end{eqnarray}
Moreover, if define $\eta_{j,0}=\frac{\hat{\eta}_{j,0}}{h_d}$ (assumption (A4)), since $h_d(t)$ is strictly decreasing
\begin{eqnarray}
\hat{\mathcal{G}}^j_{\delta,\tau}=\left\{\underset{t\in[0,\tau]}{\sup}|\hat \eta_j(t)-\hat \eta_{j,0}(t)|\leq\frac{\delta h_d(\tau)}{1+\tau},\underset{t>\tau}{\sup}|\hat \eta_j(t)-\hat \eta_{j,0}(t)|\leq\frac{1}{3}\right\}\subseteq\mathcal{G}^j_{\delta,\tau},\label{eq: Gepsilontau}
\end{eqnarray}
for all $\tau\geq\tau_{S_j}>1$.
As we may have a different $\tau_{S_j}$ for each $j\in\{0,\ldots,d\}$, we consider
\begin{eqnarray}
\tau_S&=&\underset{j\in\{0,\ldots,d\}}{\max}\{\tau_{S_j}\},\label{eqn: tauS}
\end{eqnarray} 
which guarantees that equation \eqref{eq: Gepsilontau} is satisfied for all $j$ (at the same time) when $\tau>\tau_S$. We continue by computing the probabilities of the sets $\hat{\mathcal{G}}^j_{\delta,\tau}$, which lead us to compute a centred small probability version of the event of interest.
\begin{lemma}\label{Lemma: Psotivity Joint}
Define the centred event
\begin{eqnarray}
C^j_{\delta,\tau}&=&\left\{\underset{t\in[0,\tau]}{\sup}|\hat \eta_j(t)|\leq\frac{\delta h_d(\tau)}{2(1+\tau)}, \underset{t>\tau}{\sup}|\hat \eta_j(t)|\leq\frac{1}{6}\right\},\nonumber
\end{eqnarray}
then, there exits $\tau_C>1$ such that 
\begin{eqnarray}
\Pi(C^j_{\delta,\tau})&>&0.\nonumber
\end{eqnarray}
for all $\tau\geq\tau_C$ and all $j\in\{0,\ldots,d\}$.
\end{lemma}
It follows the main result of this section 
\begin{theorem}\label{Lemma: Positivity of B}
Consider $\hat{\mathcal{G}}^j_{\delta,\tau}$ as defined in equation \eqref{eq: Gepsilontau}, then there exists $\tau_P>\tau_C$ such that
\begin{eqnarray}
\Pi\left(\hat{\mathcal{G}}^j_{\delta,\tau}\right)&>&0.\nonumber
\end{eqnarray}
for all $\tau>\tau_P$ for all $j$.
\end{theorem}

\begin{proof}[Proof of Theorem \ref{Lemma: Positivity of B}]
By Lemma \ref{Lemma: bounded supremum}, the processes $(\hat{\eta}_j(t))_{t\geq 0}$ can be seen as probability measures on the separable Banach space $\mathbb{B}=(\mathcal{C}_0,\|\|_{\infty})$. Hence, by Lemma 5.1 of Van der Vaart and van Zanten (2008) \cite{van2008reproducing}, the support of $(\hat{\eta}_j)_{t\geq 0}$ is equal to the closure of the reproducing kernel Hilbert space  $\mathbb{H}_j\subset \mathbb{B}$ in the Banach space $\mathbb{B}$.

Let $\delta_B=\frac{\delta h_d(\tau)}{2(1+\tau)}$ and take $\tau_B>0$ such that $0<\delta_B<\frac{1}{6}$ for all $\tau>\tau_B$. Notice this is possible as $h_d(t)$ is a strictly decreasing for $t>1$. Consider $g_j\in\mathbb{H}_j$ with $\left\|g_j-\hat{\eta}_{0,j}\right\|_{\infty}\leq\delta_B$, by triangular inequality,
\begin{eqnarray}
\left\|\hat{\eta}_j-\hat{\eta}_{j,0}\right\|_{\infty}&\leq&\left\|\hat{\eta}_j-g_j\right\|_{\infty}+\delta_B,
\end{eqnarray}
hence 
\begin{eqnarray}
\Pi\left(\hat{\mathcal{G}}^j_{\delta,\tau}\right)&=&
\Pi\left(\underset{t\in[0,\tau]}{\sup}|\hat{\eta}_j-\hat{\eta}_{j,0}|(t)\leq\frac{\delta h_d(\tau)}{1+\tau},\underset{t>\tau}{\sup}|\hat{\eta}_j-\hat{\eta}_{j,0}|(t)\leq\frac{1}{3}\right)\nonumber\\
&\geq& \Pi\left(\underset{t\in[0,\tau]}{\sup}|\hat{\eta}_j-g_j|(t)\leq\frac{\delta h_d(\tau)}{2(1+\tau)},\underset{t>\tau}{\sup}|\hat{\eta}_j-g_j|(t)\leq\frac{1}{3}-\delta_B\right)\nonumber\\
&\geq&e^{-\frac{1}{2}\left\|g_j\right\|_{\mathbb{H}_j}^2}\Pi\left(\underset{t\in[0,\tau]}{\sup}|\hat{\eta}_j|(t)\leq\frac{\delta h_d(\tau)}{2(1+\tau)},\underset{t>\tau}{\sup}|\hat{\eta}_j|(t)\leq\frac{1}{6}\right),\nonumber
\end{eqnarray}

where last inequality comes from the result of Kuelbs, Li and Lindi (1994) \cite{kuelbs1994gaussian} (Lemma 5.2 in Van der Vaart and van Zanten (2008) \cite{van2008reproducing}) and since $\delta_B\leq\frac{1}{6}$. 

On the other hand, by Lemma \ref{Lemma: Psotivity Joint}, there exits $\tau_C>1$ such that for all $\tau\geq\tau_C$,
\begin{eqnarray}
\Pi\left(\underset{t\in[0,\tau]}{\sup}|\hat{\eta}_j|(t)\leq\frac{\delta h_d(\tau)}{2(1+\tau)},\underset{t>\tau}{\sup}|\hat{ \eta_j}|(t)\leq\frac{1}{6}\right)>0,
\end{eqnarray}
for all $j\in\{0,\ldots,d\}$. Finally, we conclude
\begin{eqnarray}
\Pi(\mathcal{G}^j_{\delta,\tau})>0.\nonumber
\end{eqnarray}
for all $\tau>\tau_P=\max\{\tau_B,\tau_C\}$ and for all $j\in\{0,\ldots,d\}$.
\end{proof}

Given $\epsilon>0$, by Lemma \ref{Lemma: KL1} (or alternatively Lemma \ref{Lemma: KL2} if we are in the fixed covariates design), there exits $\delta(\epsilon)>0$ and $\tau(\epsilon)>1$ such that for all $\tau>\tau(\epsilon)$, the set $B_{\delta(\epsilon),\tau}$ is contained in a Kullback-Leibler neighbourhood of $\theta_0$ of radius $\epsilon$. On the other hand, by equation \eqref{eqn: tauS} and Lemma \ref{Lemma: Positivity of B}, for the same $\delta(\epsilon)>0$, there exists $\tau_S$ and $\tau_P$ such that for all $\tau>\max\{\tau_S,\tau_P\}$, $\Pi(B_{\delta(\epsilon),\tau})>0$. Therefore, we conclude the proof of condition (G1) (or (T1) respectively) by taking any set $B_{\delta(\epsilon),\tau}$ with $\tau>\{\tau_S,\tau_P,\tau(\epsilon)\}>1$.

\section{Discussion}
We have proved almost surely posterior consistency for a metric that considered  survival functions over a large class of possible stratifications of the covariate space $[0,1]^d$. This was done under the further assumption that covariates arise in a random design, i.e. they follow a probability distribution $Q$ on $[0,1]^d$, and under a non-random design.
 
In general, we found that the assumptions listed in section \ref{section: assumptions} were quite reasonable and natural for the type of data we were dealing with. 
In particular, we were able to prove the consistency result for stationary kernels $\kappa$, such that $(\kappa_j(0)-\kappa_j(2^{-n}))^{-1}\geq n^6$. We believe this is quite general as it includes a large family of kernels. Some important examples are the Squared exponential and the Ornstein-Uhlenbeck kernel. Furthermore, assumptions $(A3)$ and $(A3')$ imposed some restrictions over the tail of the distribution of the random variable we were considering. We believe they are not too restrictive as it is not unusual to consider random variables with finite second moment.

Additionally, there are two specific assumptions which come implicitly from the definition of the model in equation \eqref{eqn:model}, and may have a stronger effect for our results. We review them and discuss up to what extent they could be modified. In section 2, we inherit the specific assumption that the link function $\sigma$ is defined as the sigmoid function. We believe it should be possible to generalize this function to an arbitrary bounded and Lipchitz function since our proof just uses these particular properties of the sigmoid function. With respect the covariates, there is also a strong assumption as we assume $X$ takes values on the compact space $[0,1]^d$. While considering something more general as $\mathbb{R}^d$ will not change at all the results for the sequence of tests, it would require to completely modify the proofs regarding the Kullback-Leibler neighbourhoods in Lemma \ref{Lemma: KL1} and \ref{Lemma: KL2}. 

Lastly, it is worth noticing, that a more general version of the model including interactions between covariates could be considered, for example
\begin{eqnarray}
\lambda_x(t)&=&\Omega\sigma\left(\eta_0(t)+\sum_{j=1}^dx_j\eta_j(t)+\sum_{j=1}^d\sum_{i=j+1}^dx_jx_i\eta_{j,i}(t)\right).\nonumber
\end{eqnarray}
It is not hard to check the same techniques apply for this case, of for even more complex interactions.

\section{Acknowledgements}
The authors thanks Nicol\'as Rivera for comments that greatly improved  this manuscript. We also thank Judith Rousseau for her comments at early stage of this research.
\newpage
\section{Appendix}

\begin{proof}[Proof of Lemma \ref{Lemma: Testing Conditions without censoring} and \ref{Lemma: Testing Conditions without censoring Qn}]

We proceed to prove Lemma \ref{Lemma: Testing Conditions without censoring} and Lemma \ref{Lemma: Testing Conditions without censoring Qn}. Let $\epsilon>0$, by a direct application of the bounded  difference inequality \cite[Theorem 2.2]{devroye2012combinatorial}, we obtain 
\begin{eqnarray}
\Prob_{\theta_0}\left(\left|\underset{A\in\mathcal{A}}{\sup}|\mu_n(A)-\mu_{\theta_0}(A)|-\mathbb{E}_{\theta_0}\left(\underset{A\in\mathcal{A}}{\sup}|\mu_n(A)-\mu_{\theta_0}(A)|\right)\right|>\frac{\epsilon}{2}\right)&\leq&2e^{-n\frac{\epsilon^2}{2}}.\label{Expectation Shatter}
\end{eqnarray}
For the case of Lemma \ref{Lemma: Testing Conditions without censoring Qn}, the same result can be applied, since the function  
\begin{eqnarray}
g(T_1,\ldots, T_n)&=&\underset{A=A_1\times A_2\in\mathcal{A}}{\sup}\left|\mu_n(A)-\int_{A_2}\mu^x_{\theta_0}(A_1)Q_n(dx)\right|\nonumber\\
&=&\underset{A=A_1\times A_2\in\mathcal{A}}{\sup}\left|\sum_{i=1}^n\frac{\delta_{x_{i}}(A_2)}{n}\left(\delta_{T_i}(A_1)-\mu^{x_i}_{\theta_0}(A_1)\right)\right|,\nonumber
\end{eqnarray}
changes by at most $1/n$ when changing $T_i$, for a fixed $(x_{i})_{i\geq 1}$ and regardless of what $\mathcal{A}$ is. 
Therefore, by the bounded difference inequality, and for sets $A=A_1\times A_2 \in\mathcal{A}$
{\small{
\begin{eqnarray}
\tilde{\Prob}_{\theta_0}\left(\underset{A\in\mathcal{A}}{\sup}\left|\mu_n(A)-\int_{A_2}\mu^x_{\theta_0}(A_1)Q_n(dx)\right|-\tilde{\E}_{\theta_0}\left(\underset{A\in\mathcal{A}}{\sup}\left|\mu_n(A)-\int_{A_2}\mu_{\theta_0}^x(A_1)Q_n(dx)\right|\right)\right)&\leq&2e^{-n\frac{\epsilon^2}{2}}.\label{Expectation Shatter 2}\end{eqnarray}
}}
Going back to Lemma \ref{Lemma: Testing Conditions without censoring}, the result in \cite[Theorem 3.1]{devroye2012combinatorial} gives a bound for the expectation in equation \eqref{Expectation Shatter}, in terms of the VC shatter coefficient $S_{\mathcal{A}}(n)$ defined in equation \eqref{VC}. Indeed,
\begin{eqnarray}
\E_{\theta_0}\left(\underset{A\in\mathcal{A}}{\sup}|\mu_n(A)-\mu_{\theta_0}(A)|\right)&\leq&2\sqrt{\frac{\log 2S_{\mathcal{A}}(n)}{n}}.\nonumber
\end{eqnarray}
For the case of Lemma \ref{Lemma: Testing Conditions without censoring Qn}, the proof of \cite[Theorem 3.1]{devroye2012combinatorial} can be replicated for obtaining the same bound for the expectation in equation \eqref{Expectation Shatter 2} for a fixed sequence of covariates $(x_{i})_{i\geq 1}$, then
\begin{eqnarray}
\tilde{\E}_{\theta_0}\left(\underset{A\in\mathcal{A}}{\sup}\left|\mu_n(A)-\int_{A_2}\mu^x_{\theta_0}(A_1)Q_n(dx)\right|\right)
&\leq&2\sqrt{\frac{\log 2S_{\mathcal{A}}(n)}{n}}.\nonumber
\end{eqnarray}
As in both cases we have exactly the same bounds in terms of the shatter coefficient of $S_{\mathcal{A}}(n)$, we just prove Lemma \ref{Lemma: Testing Conditions without censoring} and argue the proof for Lemma \ref{Lemma: Testing Conditions without censoring Qn} is exactly the same. 

By \cite[Lemma 4.1]{devroye2012combinatorial}, the VC dimension of the class of subsets $\mathcal{A}$ (rectangles in $\mathbb{R}^{+}\times[0,1]^d$) is upper bounded by $2(d+1)$. Furthermore, \cite[Corollary 4.1]{devroye2012combinatorial} provides a bound for the VC shatter coefficient of the class $\mathcal{A}$ given by
\begin{eqnarray}
S_{\mathcal{A}}(n)&\leq&(n+1)^{2(d+1)}.\nonumber
\end{eqnarray}   
Combining the last results, we get
\begin{eqnarray}
\mathbb{E}_{\theta_0}\left(\underset{A\in\mathcal{A}}{\sup}|\mu_n(A)-\mu_{\theta_0}(A)|\right)&\leq&2\sqrt{\frac{\log(2(n+1)^{2(d+1)})}{n}},\nonumber
\end{eqnarray}
which decreases to 0 as n goes to infinity. From this result, we conclude there exists $N>0$ large enough such that for all $n\geq N$
\begin{eqnarray}
\mathbb{E}_{\theta_0}(\phi_n)=\Prob_{\theta_0}\left(\underset{A\in\mathcal{A}}{\sup}|\mu_n(A)-\mu_{\theta_0}(A)|>\frac{\epsilon}{4}\right)\leq 2e^{-n\frac{\epsilon^2}{2}}.\nonumber
\end{eqnarray}
We proceed to prove the exponentially small bound for the type II errror. By definition of the supremum there exists a rectangle $A^\star\in\mathcal{A}$ such that $|\mu_{\theta_0}(A^\star)-\mu_{\theta_1}(A^\star)|>\frac{2\epsilon}{3}$, with $\theta_1\in\{\theta:d(\theta,\theta_0)>\epsilon\}$. On the other hand,
\begin{eqnarray}
\E_{\theta_1}(1-\phi_n)&=&\Pr_{\theta_1}\left(\underset{A\in\mathcal{A}}{\sup}\left|\mu_n(A)-\mu_{\theta_0}(A)\right|\leq\frac{\epsilon}{4}\right)\leq\Pr_{\theta_1}\left(\left|\mu_n(A^\star)-\mu_{\theta_0}(A^\star)\right|\leq\frac{\epsilon}{4}\right).\label{eq: Fn-F0epsilon}
\end{eqnarray}
Define the event,
\begin{eqnarray}
E=\left\{\underset{A\in\mathcal{A}}{\sup}\left|\mu_n(A)-\mu_{\theta_1}(A)\right|\leq\frac{\epsilon}{4}\right\},\nonumber
\end{eqnarray}
we proceed by computing the probability of equation \eqref{eq: Fn-F0epsilon} using total probability,
\begin{eqnarray}
\eqref{eq: Fn-F0epsilon}&=&\Pr_{\theta_1}\left(\left|\mu_n(A^\star)-\mu_{\theta_0}(A^\star)\right|\leq\frac{\epsilon}{4}\cap E\right)+\Pr_{\theta_1}\left(\left|\mu_n(A^\star)-\mu_{\theta_0}(A^\star)\right|\leq\frac{\epsilon}{4}\cap E^c\right)\nonumber\\
&\leq&\Pr_{\theta_1}\left(\left|\mu_n(A^\star)-\mu_{\theta_0}(A^\star)\right|\leq\frac{\epsilon}{4}\cap \underset{A\in\mathcal{A}}{\sup}\left|\mu_n(A)-\mu_{\theta_1}(A)\right|\leq\frac{\epsilon}{4}\right)+\Pr_{\theta_1}\left(\underset{A\in\mathcal{A}}{\sup}\left|\mu_n(A)-\mu_{\theta_1}(A)\right|>\frac{\epsilon}{4}\right)\nonumber\\
&\leq&\Prob_{\theta_1}\left(\left|\mu_n(A^\star)-\mu_{\theta_0}(A^\star)\right|\leq\frac{\epsilon}{4}\cap \left| \mu_n(A^\star)-\mu_{\theta_1}(A^\star)\right|\leq\frac{\epsilon}{4}\right)+2e^{-n\frac{\epsilon^2}{2}}.\label{eq: Fn-F0epsilon2}
\end{eqnarray}
By triangle inequality, we have
\begin{eqnarray}
|\mu_{\theta_0}-\mu_{\theta_1}|-|\mu_n-\mu_{\theta_1}|\leq |\mu_n-\mu_{\theta_0}|.\nonumber
\end{eqnarray}
Plugging in this result in equation \eqref{eq: Fn-F0epsilon2},
\begin{eqnarray}
\eqref{eq: Fn-F0epsilon2}&\leq&\Pr_{\theta_1}\left(|\mu_{\theta_0}(A^\star)-\mu_{\theta_1}(A^\star)|-|\mu_n(A^\star)-\mu_{\theta_1}(A^\star)|\leq\frac{\epsilon}{4}\cap \left|\mu_n(A^\star)-\mu_{\theta_1}(A^\star)\right|\leq\frac{\epsilon}{4}\right)+2e^{-n\frac{\epsilon^2}{2}}\nonumber\\
&\leq&\Pr_{\theta_1}\left(\emptyset\right)+2e^{-n\frac{\epsilon^2}{2}},\nonumber
\end{eqnarray}
since $|\mu_{\theta_0}(A^\star)-\mu_{\theta_1}(A^\star)|>\frac{2\epsilon}{3}$ and $|\mu_n(A^\star)-\mu_{\theta_1}(A^\star)|\leq\frac{\epsilon}{4}$.
Finally,
\begin{eqnarray}
\E_{\theta_1}(1-\phi_n)&\leq&2e^{-n\frac{\epsilon^2}{2}}.\nonumber
\end{eqnarray}
Taking the supremum over the $\{\theta_1\in\Theta:d(\theta_0,\theta_1)>\epsilon\}$ we conclude the result.
\end{proof}

\begin{lemma}\label{Lemma: Bound KL with the sup in [0,tau]}
Let $\delta>0$, $\tau>1$ and $x\in[0,1]^d$. Define
\begin{eqnarray}
B_{\delta,\tau}&=&\left\{(\eta_0,\ldots,\eta_d):\underset{t\in[0,\tau]}{\sup}|\eta_j(t)-\eta_{j,0}(t)|\leq\frac{\delta}{1+\tau}, \forall j\in\{0,\ldots,d\}\right\}\nonumber
\end{eqnarray}
and
\begin{eqnarray}
Y_x(t)&=&\eta_0(t)+\sum_{j=1}^d x_j\eta_j(t)\nonumber
\end{eqnarray}
and $Y_{x0}(t)$, when replacing $\theta$ by $\theta_0$. Let $\sigma(x)$ be the sigmoid function, then for all $\delta>0$ and $(\eta_0,\ldots,\eta_d)\in B_{\delta,\tau}$,
\begin{eqnarray}
\underset{x\in[0,1]^d}{\max}\underset{t\in[0,\tau]}{\sup}|\sigma(Y_x(t))-\sigma(Y_{x0}(t))|&\leq&\frac{d+1}{1+\tau}\delta,\label{eq :supsigma}
\end{eqnarray}
and
\begin{eqnarray}
\underset{x\in[0,1]^d}{\max}\underset{t\in[0,\tau]}{\sup}|\log\sigma(Y_x(t))-\log\sigma(Y_{x0}(t))|&\leq&\frac{d+1}{1+\tau}\delta.\label{eq :sulogpsigma}
\end{eqnarray}
\end{lemma}
\begin{proof}
Let $a=Y_x(t)$ and $b=Y_{x0}(t)$, by the mean value theorem, there exits $c$ between $a$ and $b$ such that
\begin{eqnarray}
|\sigma(a)-\sigma(b)|&=&\left|\frac{e^c}{(e^c+1)}(a-b)\right|\nonumber\\
&\leq&|a-b|\nonumber\\
&=&\left|(\eta_0-\eta_{0,0})(t)+\sum_{j=1}^dx_j(\eta_j-\eta_{j,0})(t)\right|\nonumber
\end{eqnarray}
By taking supremum on both sides of the above equation, we get equation \eqref{eq :supsigma}, indeed
\begin{eqnarray}
\underset{t\in[0,\tau]}{\sup}|\sigma(Y_{x}(t))-\sigma(Y_{x0}(t))|&\leq&\underset{t\in[0,\tau]}{\sup}\left|(\eta_0-\eta_{0,0})(t)+\sum_{j=1}^dx_j(\eta_j-\eta_{j,0})(t)\right|\nonumber\\
&\leq&\underset{t\in[0,\tau]}{\sup}\left|(\eta_0-\eta_{0,0})(t)\right|+\sum_{j=1}^d|x_j|\underset{t\in[0,\tau]}{\sup}\left|\eta_j-\eta_{j0}(t)\right|\nonumber\\
&\leq&\frac{\delta}{1+\tau}+\sum_{j=1}^d|x_j|\frac{\delta}{1+\tau}\nonumber\\
\underset{x\in[0,1]^d}{\max}\underset{t\in[0,\tau]}{\sup}|\sigma(Y_{x}(t))-\sigma(Y_{x0}(t))|
&\leq &\frac{d+1}{1+\tau}\delta\nonumber
\end{eqnarray}
The proof for equation \eqref{eq :sulogpsigma} follows by the same argument.
\end{proof} 

\begin{proof}[Proof of Lemma \ref{Lemma: KL1} and \ref{Lemma: KL2}]

The function
\begin{eqnarray}
f_x(t;\theta)&=&\Omega\sigma(Y_x(t))e^{-\int_0^t\Omega\sigma(Y_x(s))ds}\nonumber
\end{eqnarray}
can be interpreted as the conditional distribution of a survival time $T$ given the covariate $X=x$ in the random design. But it also can be interpreted as the distribution of the time $T$ for a fixed valued of the covariate $x$, in the fixed design. Additionally, we define
\begin{eqnarray}
Y_x(t)&=&\eta_0(t)+\sum_{j=1}^d x_j\eta_j(t)\nonumber
\end{eqnarray}
and $Y_{x0}(t)$, when replacing $\theta$ by $\theta_0$.

\begin{itemize}
\item[$\clubsuit$]\textbf{[Random design, condition (G1)]}\\
Using the towering property of conditional expectation, we have
\begin{eqnarray}
\E_{\theta_0}\left(\log\frac{f_X(T,\theta_0)}{f_X(T,\theta)}\right)&=&\int_{[0,1]^d}\mathbb{E}_{\theta_0}\left(\log\frac{f_x(T,\theta_0)}{f_x(T,\theta)}\given X=x\right)Q(dx).\nonumber
\end{eqnarray}
Moreover, by definition
{\small{
\begin{eqnarray}
\mathbb{E}_{\theta_0}\left(\log\frac{f_x(T,\theta_0)}{f_x(T,\theta)}\given X=x\right)&=&
\int_0^\tau\log\frac{f_x(t,\theta_0)}{f_x(t,\theta)}f_x(t,\theta_0)dt+\int_\tau^\infty\log\frac{f_x(t,\theta_0)}{f_x(t,\theta)}f_x(t,\theta_0)dt.\label{KL}
\end{eqnarray}
}}
The first integral breaks into three parts,
\begin{eqnarray}
\int_0^\tau\log\frac{f_x(t,\theta_0)}{f_x(t,\theta)}f_x(t,\theta_0)dt&\leq&\int_0^\tau\log\frac{\Omega_0}{\Omega}f_x(t,\theta_0)dt+\int_0^\tau\left|\log\frac{\sigma(Y_{x0}(t))}{\sigma(Y_{x}(t))}\right|f_x(t,\theta_0)dt\nonumber\\
&+&\int_0^\tau\left|\int_0^{t}\Omega\sigma(Y_{x}(s))-\Omega_0\sigma(Y_{x0}(s))ds\right|f_x(t,\theta_0)dt\nonumber\\
&=&I_1+I_2+I_3.\label{eq: F1}
\end{eqnarray}
We proceed to bound each of the integrals in equation \eqref{eq: F1}. For $\theta\in B_{\delta,\tau}$,
\begin{eqnarray}
I_1=\int_0^\tau\log\frac{\Omega_0}{\Omega}f_x(t,\theta_0)dt&\leq&\int_0^\tau\left(\frac{\Omega_0}{\Omega}-1\right)f_x(t,\theta_0)dt\leq\frac{\delta}{1-\delta}.\nonumber
\end{eqnarray}
For the second integral, by equation \eqref{eq :sulogpsigma},
\begin{eqnarray}
I_2=\int_0^\tau\left|\log\frac{\sigma(Y_{x0}(t))}{\sigma(Y_{x}(t))}\right|f_x(t,\theta_0)dt&\leq&\int_0^\tau\frac{d+1}{\tau+1}\delta f_x(t,\theta_0)dt<\frac{d+1}{\tau+1}\delta,\nonumber
\end{eqnarray}
since $f_x(\cdot,\theta_0)$ a p.d.f..\\
Lastly, we break the last integral of equation \eqref{eq: F1} into two terms,
\begin{eqnarray}
I_3&=&\int_0^\tau\left|\int_0^{t}\left(\Omega\sigma(Y_{x}(s))-\Omega_0\sigma(Y_{x}(s))+\Omega_0\sigma(Y_{x}(s))-\Omega_0\sigma(Y_{x0}(s))\right)ds\right|f_x(t,\theta_0)dt\nonumber\\
&\leq&\int_0^\tau\left|\int_0^{t}(\Omega\sigma(Y_{x}(s))-\Omega_0\sigma(Y_{x}(s)))ds\right|f_x(t,\theta_0)dt\nonumber\\
&+&\int_0^\tau\left|\int_0^{t}(\Omega_0\sigma(Y_{x}(s))-\Omega_0\sigma(Y_{x0}(s)))ds\right|f_x(t,\theta_0)dt\nonumber\\
&=&I_{3_{1}}+I_{3_{2}}.\nonumber
\end{eqnarray}
For the first part we use the fact that $|\Omega-\Omega_0|\leq\delta\Omega_0$ and that $\int_0^t\sigma(Y_x(t))\leq t$ (since $\sigma$ is the sigmoid function and thus upper bounded by 1), hence 
\begin{eqnarray}
I_{3,1}&\leq&\int_0^\tau|\Omega-\Omega_0|tf_x(t,\theta_0)dt\leq\delta\Omega_0\mathbb{E}_{\theta_0}(T|X=x).\nonumber
\end{eqnarray}
For the second term, by equation \eqref{eq :supsigma}
\begin{eqnarray}
I_{3,2}&\leq&\int_0^\tau\Omega_0\int_0^{t}\left|\sigma(Y_{x}(s))-\sigma(Y_{x0}(s))\right|dsf_x(t,\theta_0)dt\nonumber\\
&\leq&\int_0^\tau\Omega_0\int_0^{t}\frac{d+1}{\tau+1}\delta dsf_x(t,\theta_0)dt\nonumber\\
&\leq&\Omega_0\frac{\tau(d+1)}{\tau+1}\delta\leq\Omega_0(d+1)\delta.\nonumber
\end{eqnarray}
Putting everything together,
\begin{eqnarray}
\int_0^\tau\log\frac{f_x(t,\theta_0)}{f_x(t,\theta)}f_x(t,\theta_0)dt&\leq&I_1+I_2+I_3\nonumber\\
&\leq&\delta\left(\frac{1}{1-\delta}+\frac{d+1}{\tau+1}+\Omega_0(d+1)+\Omega_0\mathbb{E}_{\theta_0}(T|X=x)\right).\label{Th3 cond1}
\end{eqnarray}
For $\delta\in(0,1/2)$ and $\tau\geq 1$, we have
\begin{eqnarray}
\mathbb{E}\left(\int_0^\tau\log\frac{f_x(t,\theta_0)}{f_x(t,\theta)}f_x(t,\theta_0)dt\right)&\leq&\delta\left(2+\frac{d+1}{2}+\Omega_0(d+1)+\Omega_0\int_{x\in[0,1]^d}\mathbb{E}_{\theta_0}(T|X=x)Q(dx)\right).\nonumber
\end{eqnarray}
Using assumption A3 we have that the latter integral is finite, hence we can choose $\delta$ small enough such that the whole term is less than $\epsilon/2$.\\

Now take the second integral of equation \eqref{KL} and break it into two integrals,
\begin{eqnarray}
\int_\tau^\infty\log\frac{f_x(t,\theta_0)}{f_x(t,\theta)}f_x(t,\theta_0)dt&\leq&\int_\tau^\infty\log f_x(t,\theta_0)f_x(t,\theta_0)dt+\int_\tau^\infty|\log f_x(t,\theta)|f_x(t,\theta_0)dt\nonumber\\
&=&I_1+I_2.\nonumber
\end{eqnarray}
For the first integral, since $\log(x)\leq x-1$,
\begin{eqnarray}
I_1&\leq&\int_\tau^\infty(f_x(t,\theta_0)-1)f_x(t,\theta_0)dt\nonumber\\
&\leq&\int_\tau^\infty f_x(t,\theta_0)^2dt.\nonumber
\end{eqnarray}
For the second integral, we have that
\begin{eqnarray}
I_2&=&\int_\tau^\infty\left|\log \left(\Omega\sigma(Y_x(t))e^{-\Omega\int_0^t\sigma(Y_x(s))ds}\right)\right| f_x(t,\theta_0)dt\nonumber\\
&\leq&\int_\tau^\infty|\log\Omega| f_x(t,\theta_0)dt+\int_\tau^\infty|\log \sigma(Y_x(t))|f_x(t,\theta_0)dt\nonumber\\
&+&\int_\tau^\infty\left|\Omega\int_0^t\sigma(Y_x(s))ds\right|f_x(t,\theta_0)dt\nonumber\\
&=&I_{2,1}+I_{2,2}+I_{2,3}.\nonumber
\end{eqnarray}
Recall that for $\theta\in B_{\delta,\tau}$ we have that $\Omega_0(1-\delta)\leq\Omega\leq\Omega_0(1+\delta)$ and $\delta\in(0,1/2)$, hence the first integral is finite and bounded by,
\begin{eqnarray}
I_{2,1}&\leq&K_0\mathbb{P}_{\theta_0}(T>\tau|X=x),\nonumber
\end{eqnarray}
where $K_0$ is a constant depending on $\Omega_0$ and $\delta$.\\

For $\theta \in B_{\delta,\tau}$ we have that $\eta_j(t)h_d(t)>-1$ for all $t\geq\tau$ and all $j\in\{0,\ldots,d\}$. Then, by the definition of $h_d(t)$, for all $t\geq\tau$
\begin{eqnarray}
Y_x(t)\geq\sum_{j=0}^d\eta_j(t)\geq-t-\log(1-e^{-t}).\label{weird condition}
\end{eqnarray}
Furthermore, notice that the inverse $\sigma^{-1}(e^{-t})=-t-\log(1-e^{-t})$, which finally implies $0\geq\log\sigma(Y_x(t))\geq-t$ since the function $\sigma$ is upper bounded by one and covariates $X\in[0,1]^d$. Using this last argument, we get a bound for
\begin{eqnarray}
I_{2,2}&\leq&\int_\tau^\infty t f_x(t,\theta_0)dt=\mathbb{E}_{\theta_0}(T 1_{\{T>\tau\}}|X=x).\nonumber
\end{eqnarray}
Finally, for the last integral we have
\begin{eqnarray}
I_{2,3}&\leq&\int_\tau^\infty\Omega_0(1+\delta)tf_x(t,\theta_0)dt=\Omega_0(1+\delta)\mathbb{E}_{\theta_0}(T 1_{\{T>\tau\}}|X=x).\nonumber
\end{eqnarray}
Putting everything together, we have
\begin{eqnarray}
\int_\tau^\infty\log\frac{f_x(t,\theta_0)}{f_x(t,\theta)}f_x(t,\theta_0)dt&\leq&\int_\tau^\infty f_x(t,\theta_0)^2dt+K_0\mathbb{P}_{\theta_0}(T>\tau|X=x)\nonumber\\
&+&(1+\Omega_0(1+\delta))\mathbb{E}_{\theta_0}(T 1_{\{T>\tau\}}|X=x).\label{Th3 cond2}
\end{eqnarray}
By definition, $f_x(t,\theta_0)$ is bounded by $\Omega_0$ for all $t\geq 0$ and $x\in[0,1]^d$. Using this fact and by taking expectation, we get
\begin{eqnarray}
\mathbb{E}\left(\int_\tau^\infty\log\frac{f_x(t,\theta_0)}{f_x(t,\theta)}f_x(t,\theta_0)dt\right)&\leq&(\Omega_0+K_0)\mathbb{E}(\mathbb{P}_{\theta_0}(T>\tau|X=x))\nonumber\\
&+&(1+\Omega_0(1+\delta))\mathbb{E}(\mathbb{E}_{\theta_0}(T 1_{\{T>\tau\}}|X=x)).\label{integralfromtau}
\end{eqnarray}
We conclude by using the following claim
\begin{eqnarray}
\underset{n\to\infty}{\lim}\E(\E_{\theta_0}(T1_{\{T>n\}}|X=x))=0\label{prove}
\end{eqnarray}
and that for a positive random variable $T$, and $\tau\geq 1$,
\begin{eqnarray}
\E(\Prob_{\theta_0}(T>\tau|X=x))&\leq&\E(\E_{\theta_0}(T1_{\{T>\tau\}}|X=x)).\nonumber
\end{eqnarray}
We conclude that the integral of the tail is finite and hence we can choose $\tau\geq 1$ large enough, such that equation \eqref{integralfromtau} sum less than $\epsilon/2$ and
\begin{eqnarray}
\E_{\theta_0}\left(\log\frac{f_X(T,\theta_0)}{f_X(T,\theta)}\right)&\leq&\epsilon.
\end{eqnarray}
\textbf{Proof of the claim of equation \eqref{prove}}. 
Define the functions
\begin{eqnarray}
I(n)&=&\int_{[0,1]^d}\mathbb{E}_{\theta_0}(T1_{\{T\leq n\}}|X=x)Q(dx),\nonumber
\end{eqnarray}
and
\begin{eqnarray}
I^c(n)&=&\int_{[0,1]^d}\mathbb{E}_{\theta_0}(T1_{\{T> n\}}|X=x)Q(dx),\nonumber
\end{eqnarray}
and notice that
\begin{eqnarray}
\mathbb{E}_{\theta_0}(T)&=&\int_{[0,1]^d}\mathbb{E}_{\theta_0}(T|X=x)Q(dx)\nonumber\\
&=&I(n)+I^c(n)\nonumber
\end{eqnarray}
for every $n$. Taking the limit of $I(n)$ when $n$ goes to infinity, we have aim to show that
\begin{eqnarray}
\underset{n\to\infty}{\lim}I(n)&=&\int_{x\in[0,1]^d}\mathbb{E}_{\theta_0}(T|X=x)Q(dx).\nonumber
\end{eqnarray}
Notice $\mathbb{E}_{\theta_0}(T1_{\{T\leq n\}}|X=x)\leq\mathbb{E}_{\theta_0}(T|X=x)$, the last expectation is integrable by assumption (A2), hence by dominated convergence,
\begin{eqnarray}
\underset{n\to\infty}{\lim}I(n)&=&\int_{x\in[0,1]^d}\underset{n\to\infty}{\lim}\mathbb{E}_{\theta_0}(T1_{\{T\leq n\}}|X=x)Q(dx).\nonumber
\end{eqnarray}
Notice that the function $f_n(t)=t1_{\{t\leq n\}}$ is such that $f_n\leq f_{n+1}$ and converges point-wise to $f(t)=t$, then by monotone convergence theorem,
\begin{eqnarray}
\underset{n\to\infty}{\lim}I(n)&=&\int_{x\in[0,1]^d}\mathbb{E}_{\theta_0}\left(\underset{n\to\infty}{\lim}T1_{\{T\leq n\}}|X=x\right)Q(dx)=\mathbb{E}_{\theta_0}(T).\nonumber
\end{eqnarray}
Therefore, we conclude
\begin{eqnarray}
\underset{n\to\infty}{\lim}I^c(n)=0.\nonumber
\end{eqnarray}

\item[$\clubsuit$]\textbf{[Fixed design, condition (T1) (i)]}\\
Take $f_x(t,\theta_0)$ as the distribution of times under the true parameter $\theta_0$ and covariate $x$, then
\begin{eqnarray}
K_i(\theta_0,\theta)&=&\mathbb{E}\left(\Upsilon(\theta_0;\theta)\right)\nonumber\\
&=&\int_0^\infty\log\frac{f_x(t,\theta_0)}{f_x(t,\theta)}f_x(t,\theta_0)dt.\nonumber
\end{eqnarray}
which has the same form as equation \eqref{KL}, hence we replicate exactly the same proof of the random design case, but we justify with assumption (A3') instead of assumption (A3). Indeed, notice that the main assumption we are using to conclude the result in the random design case is that $\E_Q(\E_{\theta_0}(T|X))<M<\infty$. Since in the fixed design case we are working with expressions of the form $\E_{\theta_0}(T|x)$ (without integrating over the covariates), the uniformly integrability required in assumption (A3') will be a sufficient condition to replace the latter assumption.  
\item[$\clubsuit$]\textbf{[Fixed design, condition (T1) (ii)]}\\
Using the definition of variance,
\begin{eqnarray}
V_i(\theta_0,\theta)=\mathbb{V}ar_{\theta_0}(\Upsilon(\theta_0,\theta))\leq\mathbb{E}_{\theta_0}(\Upsilon^2(\theta_0,\theta))\nonumber
\end{eqnarray}
By definition 
\begin{eqnarray}
\mathbb{E}_{\theta_0}(\Upsilon^2(\theta_0,\theta))&=&\int_0^\tau\log^2\frac{f_x(t,\theta_0)}{f_x(t,\theta)}f_x(t,\theta_0)dt+
\int_\tau^\infty\log^2\frac{f_x(t,\theta_0)}{f_x(t,\theta)}f_x(t,\theta_0)dt\nonumber\\
&=&I_1+I_2\label{EqVariance1}
\end{eqnarray}
Consider the first integral of equation \eqref{EqVariance1}, replacing the p.d.f. we have
\begin{eqnarray}
I_1&=&\int_0^\tau\left(\log\frac{\Omega_0}{\Omega}+\left|\log\frac{\sigma(Y_{x0}(t))}{\sigma(Y_{x}(t))}\right|+\left|\int_0^{t}\Omega\sigma(Y_{x}(s))-\Omega_0\sigma(Y_{x0}(s))ds\right|\right)^2 f_x(t,\theta_0)dt\nonumber
\end{eqnarray}
Notice that for all $a,b\in\mathbb{R}$ hold that $(a+b)^2\leq 2(a^2+b^2)$, we use this to bound the above quantity by
\begin{eqnarray}
I_1&\leq&\int_0^\tau 3\left(\log^2\frac{\Omega_0}{\Omega}+\left|\log\frac{\sigma(Y_{x0}(t))}{\sigma(Y_{x}(t))}\right|^2+\left|\int_0^{t}\Omega\sigma(Y_{x}(s))-\Omega_0\sigma(Y_{x0}(s))ds\right|^2\right) f_x(t,\theta_0)dt\nonumber\\
&=&\int_0^\tau 3\log^2\frac{\Omega_0}{\Omega}f_x(t,\theta_0)dt+\int_0^\tau 3\left|\log\frac{\sigma(Y_{x0}(t))}{\sigma(Y_{x}(t))}\right|^2 f_x(t,\theta_0)dt\nonumber\\
&+&\int_0^\tau 3\left|\int_0^{t}\Omega\sigma(Y_{x}(s))-\Omega_0\sigma(Y_{x0}(s))ds\right|^2 f_x(t,\theta_0)dt\nonumber\\
&=&I_{1,1}+I_{1,2}+I_{1,3}.\label{EqVariance2}
\end{eqnarray}
Take the first integral of equation \eqref{EqVariance2}. Since $\theta\in B_{\delta,\tau}$ and $\delta\in(0,1/2)$ we have
\begin{eqnarray}
I_{1,1}&\leq&\int_0^\tau3\left(\frac{\Omega_0}{\Omega}-1\right)^2f_x(t,\theta_0)dt\leq12\delta,\nonumber
\end{eqnarray}
since for $\Omega\in B_{\delta,\tau}$, $\frac{\Omega_0}{\Omega}-1\leq\frac{\delta}{1-\delta}\leq 2\delta$. Next, consider the second integral of equation \eqref{EqVariance2}. By equation \eqref{eq :sulogpsigma} and since $\tau\geq 1$ 
\begin{eqnarray}
I_{1,2}&\leq&\int_0^\tau3\left(\frac{d+1}{1+\tau}\delta\right)^2f_x(t,\theta_0)dt\leq\frac{3(d+1)^2}{2}\delta\nonumber
\end{eqnarray}
Lastly, we take the last integral of equation \eqref{EqVariance2}, using $(a+b)^2\leq 2(a^2+b^2)$,
\begin{eqnarray}
I_{1,3}&=&\int_0^\tau 3\left|\int_0^{t}\Omega\sigma(Y_{x}(s))-\Omega_0\sigma(Y_{x}(s))+\Omega_0\sigma(Y_{x}(s))-\Omega_0\sigma(Y_{x0}(s))ds\right|^2 f_x(t,\theta_0)dt\nonumber\\
&\leq&\int_0^\tau6\left|\int_0^{t}\Omega\sigma(Y_{x}(s))-\Omega_0\sigma(Y_{x}(s))ds\right|^2f_x(t,\theta_0)dt\nonumber\\
&+&\int_0^\tau6\left|\int_0^{t}\Omega_0\sigma(Y_{x}(s))-\Omega_0\sigma(Y_{x0}(s))ds\right|^2f_x(t,\theta_0)dt\nonumber\\ 
&=&I_{1,3,1}+I_{1,3,2}\nonumber
\end{eqnarray}
Given that $\theta\in B_{\delta,\tau}$, we have that $|\Omega-\Omega_0|\leq\Omega_0\delta$. Moreover, recall that $\int_0^t\sigma(Y_x(s))ds\leq t$ with probability one for all $x\in[0,1]^d$, hence
\begin{eqnarray}
I_{1,3,1}&\leq&K_0\delta\int_0^\tau t^2 f_x(t,\theta_0)dt\leq K_0\delta\mathbb{E}_{\theta_0}(T^2|x)\label{Th3 cond3}
\end{eqnarray}
for some constant $K_0$.\\ 

For the second term we have, by equation \eqref{eq :supsigma},
\begin{eqnarray}
I_{1,3,2}&\leq&\int_0^\tau 6\Omega_0^2\left(\int_0^t\left|\sigma(Y_{x}(s))-\sigma(Y_{x0}(s))\right|ds
\right)^2f_x(t,\theta_0)dt\nonumber\\
&\leq&\int_0^\tau 6\Omega_0^2\left(\int_0^t\frac{d+1}{1+\tau}\delta ds
\right)^2f_x(t,\theta_0)dt\nonumber\\
&\leq&K_0'\delta\nonumber
\end{eqnarray}
for some constant $K_0'$.\\

Putting everything together, and using assumption (A3') we have
\begin{eqnarray}
\int_0^\tau\log^2\frac{f_x(t,\theta_0)}{f_x(t,\theta)}f_x(t,\theta_0)dt&=&\delta\left(12+\frac{3(d+1)^2}{2}+K_0\E_{\theta_0}(T^2|x)+K_0'\right)<\infty\nonumber
\end{eqnarray}
Now, we take the second integral of equation \eqref{EqVariance1}
\begin{eqnarray}
I_2&\leq&\int_\tau^\infty\left(\left|\log f_x(t,\theta_0)\right|+\left|\log f_x(t,\theta)\right|\right)^2 f_x(t,\theta_0)dt\nonumber\\
&\leq&\int_\tau^\infty2\left|\log f_x(t,\theta_0)\right|^2 f_x(t,\theta_0)dt+\int_\tau^\infty2\left|\log f_x(t,\theta)\right|^2 f_x(t,\theta_0)dt\nonumber\\
&=&I_{2,1}+I_{2,2}\label{EqVariance3}
\end{eqnarray}
Checking the first integral of equation \eqref{EqVariance3} we have,
\begin{eqnarray}
I_{2,1}&\leq&\int_\tau^\infty2\left|f_x(t,\theta_0)-1\right|^2f_x(t,\theta_0)dt\nonumber\\
&\leq&\int_\tau^\infty2f_x(t,\theta_0)^3dt-\int_\tau^\infty4f_x(t,\theta_0)^2dt+2\nonumber\\
&\leq&\int_\tau^\infty2f_x(t,\theta_0)^3dt+2\nonumber
\end{eqnarray} 
Recall that $f_x(t,\theta_0)$ is bounded by $\Omega_0$ for all $x\in[0,1]^d$, hence 
\begin{eqnarray}
I_{2,1}&\leq&\int_\tau^\infty2\Omega_0^2f_x(t,\theta_0)dt+2\leq2\Omega_0^2+2<\infty\nonumber
\end{eqnarray}
For the second integral, we have
\begin{eqnarray}
I_{2,2}&\leq&\int_\tau^\infty2|\log \Omega\sigma(Y_x(t))e^{-\int_0^t\Omega\sigma(Y_x(s))ds}|^2f_x(t,\theta_0)dt\nonumber\\
&\leq&\int_\tau^\infty 6|\log\Omega|^2 f_x(t,\theta_0)dt+\int_\tau^\infty 6|\log\sigma(Y_x(t))|^2f_x(t,\theta_0)dt\nonumber\\
&+&\int_\tau^\infty 6\left(\int_0^t\Omega\sigma(Y_x(s))ds\right)^2 f_x(t,\theta_0)dt\label{EqVariance4}
\end{eqnarray}
Recall that for $\theta\in B_{\delta,\tau}$ we have $\Omega_0(1-\delta)\leq\Omega\leq\Omega_0(1+\delta)$, hence the first integral of equation \eqref{EqVariance4} is finite. Using the same argument as before, equation \eqref{weird condition}, for $\theta\in B_{\delta,\tau}$ we have $-t<\log\sigma(Y_x(t))\leq 0$ for all $x\in[0,1]^d$. Hence, the second integral of equation \eqref{EqVariance4}
\begin{eqnarray}
\int_\tau^\infty 6\log^2\sigma(Y_x(t))f_x(t,\theta_0)dt&\leq&6\mathbb{E}_{\theta_0}(T^2|x)\label{Th3 cond4}
\end{eqnarray}
Lastly, we bound the last integral of equation \eqref{EqVariance4} by,
\begin{eqnarray}
\int_\tau^\infty 6\left(\int_0^t\Omega\sigma(Y_x(s))ds\right)^2 f_x(t,\theta_0)dt&\leq&6K_2\mathbb{E}_{\theta_0}(T^2|x)\label{Th3 cond5}
\end{eqnarray}
where $K_2$ is some constant that depends on $\Omega_0$. Finally, by the uniformly integrability of $T^2$ over $x\in[0,1]^d$, assumption $(A3')$
\begin{eqnarray}
V_i(\theta_0,\theta)&<&\infty
\end{eqnarray}
uniformly in $i$.
\end{itemize}
\end{proof}

\begin{proof}[Proof of Lemma \ref{Lemma: separable}]
Define the set
\begin{eqnarray}
D_{0}^n=\left\{g:\mathbb{R}^{+}\to\mathbb{R}: g(t)=
\left\{\begin{array}{cc}
g(t)&t\in[0,n]\\
\frac{g(n)}{t+1-n}&t\geq n
\end{array}\right. \right\},\nonumber
\end{eqnarray}
where $g(x)\in\mathcal{C}[0,n]$, the set of continuous functions on the compact $[0,n]$. Furthermore, define $D_0$ as the countable union of these sets,
\begin{eqnarray}
D_0&=&\bigcup_{n\geq 1}D_{0}^n.\nonumber
\end{eqnarray}
Take $f\in \mathcal{C}_0$. By definition, given $\delta>0$, there exits $T>0$ such that for all $t\geq T$,
\begin{eqnarray}
\underset{t\geq T}{\sup}|f(t)|\leq \delta.\nonumber
\end{eqnarray}
Let $\epsilon>0$, and define, 
\begin{eqnarray}
g_N(t)=\left\{\begin{array}{cc}
f(t)&t\in[0,N]\\
\frac{f(N)}{t+1-N}& x\geq N
\end{array}\right.,\nonumber
\end{eqnarray}
where $N$ is any natural number such that for $N>T^{\star}$ and $T^{\star}>0$ we have that
$\underset{t\geq T^\star}{\sup}|f(t)|\leq \frac{\epsilon}{3}$. Notice that by construction $g_N$ belongs to $D_0$,
\begin{eqnarray}
\underset{t\in\mathbb{R}^{+}}{\sup}|f(t)-g_N(t)|&=&\underset{t\geq N}{\sup}|f(t)-g_N(t)|\nonumber\\
&\leq&\underset{t\geq N}{\sup}|f(t)|+|g_N(t)|\leq \frac{2}{3}\epsilon.\nonumber
\end{eqnarray}
In order to complete the argument, we need to prove that $D_0^n$ is separable. This fact along with the fact that a countable union of countable sets is countable completes the proof.\\

It is well known that the space $\mathcal{C}[0,n]$ endowed with the uniform norm is separable. Then we have a countable collections of functions $\{h_m^n\}_{m=1}^\infty$ that approximate any element on $\mathcal{C}[0,n]$. Consider the sequence $\{h_{0,m}^n\}_{m=1}^\infty$ defined as,
\begin{eqnarray}
h_{0,m}^n=\left\{\begin{array}{cc}
h_m^n(t)&t\in[0,n]\\
\frac{h_m^n(n)}{t+1-n}&t\geq n
\end{array}\right..\nonumber
\end{eqnarray}
By construction, for any $g_n\in D_0^n$ and for all $\epsilon>0$ there exists $h^n_{0,m}$ such that,
\begin{eqnarray}
\underset{t\in[0,n]}{\sup}|g_n-h^n_{0,m}|=\underset{t\in[0,n]}{\sup}|g_n-h^n_{m}|<\epsilon.\nonumber
\end{eqnarray}
On the other hand, the latter implies $|g_n(n)-h_{0,m}(n)|<\epsilon$. Hence
\begin{eqnarray}
\underset{t\geq n}{\sup}\left|g_n(t)-h^n_{0,m}(t)\right|=\underset{t\geq n}{\sup}\left|\frac{g_n(n)}{t+1-n}-\frac{h^n_{m}(n)}{t+1-n}\right|\leq|g_n(n)-h_{0,m}(n)|<\epsilon.\nonumber
\end{eqnarray}
Therefore
\begin{eqnarray}
\underset{t\in\mathbb{R}^{+}}{\sup}|g_n-h^n_{0,m}|<\epsilon,\nonumber
\end{eqnarray}
concluding that $D_0^n$ is separable.
\end{proof}

\begin{lemma}\label{Lemma:sup ineq 1}Let $\tau>0$ and $\eta(t)$ be a continuous function on the interval $[0,\tau]$, then for $t^n_k=\frac{k\tau}{2^n}$, with $k\in\{0,\ldots,2^n\}$ and $n\geq 1$
\begin{eqnarray}
\underset{t\in[0,\tau]}{\sup}|\eta_t|\leq|\eta(0)|+|\eta(\tau)|+\sum_{n=1}^\infty\underset{0\leq k\leq 2^n-1}{\max}|\eta(t^n_{k+1})-\eta(t^n_k)|.\nonumber
\end{eqnarray}
\end{lemma}
\begin{proof}
Define a partition of the interval $[0,\tau]$ as $A_n=\{t^n_k=\frac{k\tau}{2^n},k=0,\ldots,2^n\}$.
Let $A_0=\{0,\tau\}$, then
\begin{eqnarray}
\underset{t\in A_0}{\sup}|\eta(t)|&\leq&|\eta(0)|+|\eta(\tau)|.\nonumber
\end{eqnarray}
For $A_1=\{0,\tau/2,\tau\}$, we want to show that
\begin{eqnarray}
\underset{t\in A_1}{\sup}|\eta(t)|&\leq&|\eta(0)|+|\eta(\tau)|+\max\{|\eta(\tau)-\eta(\tau/2)|,|\eta(\tau/2)-\eta(0)|\}.\nonumber
\end{eqnarray}
If the supremum is at $\{0,\tau\}$ we are done. Assume the supremum is at $\tau/2$, and let $x\in\{0,\tau\}$
{\footnotesize{
\begin{eqnarray}
|\eta(\tau/2)|=|\eta(\tau/2)-\eta(x)+\eta(x)|\leq|\eta(\tau/2)-\eta(x)|+|\eta(x)|\leq\max\{|\eta(\tau)-\eta(\tau/2)|,|\eta(\tau/2)-\eta(0)|\}+|\eta(x)|,\nonumber
\end{eqnarray}
}}
then the result follows. By induction, assume this holds for the partition $A_n$, so that
\begin{eqnarray}
\underset{t\in A_n}{\sup}|\eta(t)|\leq|\eta(0)|+|\eta(\tau)|+\sum_{i=1}^n\underset{0\leq k\leq 2^i-1}{\max}|\eta(t^i_{k+1})-\eta(t^i_k)|,\nonumber
\end{eqnarray}
then we need to prove this result for $A_{n+1}$. Explicitly, we want to show that
\begin{eqnarray}
\underset{t\in A_{n+1}}{\sup}|\eta(t)|&\leq&|\eta(0)|+|\eta(\tau)|+\sum_{i=1}^{n+1}\underset{0\leq k\leq 2^i-1}{\max}|\eta(t^i_{k+1})-\eta(t^i_k)|,\nonumber
\end{eqnarray}
i.e.
{\footnotesize{
\begin{eqnarray}
\max\left\{\underset{t\in A_{n}}{\sup}|\eta(t)|,\underset{t\in A_{n+1}/A_n}{\sup}|\eta(t)|\right\}&\leq&|\eta(0)|+|\eta(\tau)|+\sum_{i=1}^{n}\underset{0\leq k\leq 2^i-1}{\max}|\eta(t^i_{k+1})-\eta(t^i_k)|+\underset{0\leq k\leq 2^{n+1}-1}{\max}|\eta(t^{n+1}_{k+1})-\eta(t^{n+1}_k)|.\nonumber
\end{eqnarray}
}}
If the supremum is at $A_n$ we are done. Assume the supremum is at $A_{n+1}/A_n=\left\{\frac{(2k+1)\tau}{2^{n+1}}, k=0,\ldots,2^{n}-1\right\}$. Let the supremum be at $\frac{(2m+1)\tau}{2^{n+1}}$ and let be $x\in\left\{\frac{2m\tau}{2^{n+1}},\frac{(2m+2)\tau}{2^{n+1}}\right\}\subset A_n$ with $m\in\{0,\ldots,2^{n}-1\}$ then,
\begin{eqnarray}
\left|\eta\left(\frac{(2m+1)\tau}{2^{n+1}}\right)\right|&\leq& \left|\eta\left(\frac{(2m+1)\tau}{2^{n+1}}\right)-\eta(x)\right|+|\eta(x)|\nonumber\\
&\leq&\underset{0\leq k\leq 2^{n+1}-1}{\max}|\eta(t^{n+1}_{k+1})-\eta(t^{n+1}_k)|+|\eta(0)|+|\eta(\tau)|+\sum_{i=1}^{n}\underset{0\leq k\leq 2^i-1}{\max}|\eta(t^i_{k+1})-\eta(t^i_k)|.\nonumber
\end{eqnarray} 
Finally, since $\eta$ is continuous and $[0,\tau]$ is compact, there exits a $x\in[0,\tau]$ such that $\underset{t\in[0,\tau]}{\sup}|\eta(t)|=\eta(x)$. Then, for all $\epsilon>0$ it exits $n$ and $k\in\{0,\ldots,2^n\}$ such that
\begin{eqnarray}
\eta(x) \leq \eta\left(\frac{k\tau}{2^n}\right)+\epsilon,\nonumber
\end{eqnarray}
since
\begin{eqnarray}
\eta\left(\frac{k\tau}{2^n}\right)\leq|\eta(0)|+|\eta(\tau)|+\sum_{i=1}^\infty\underset{0\leq k\leq 2^i-1}{\max}|\eta(t^i_{k+1})-\eta(t^i_k)|,\nonumber
\end{eqnarray}
we obtained the desired conclusion.
\end{proof}

\begin{theorem}[Gaussian Correlation Inequality \cite{memarian2013gaussian}]\label{GaussianCorrelation}
For any $n\geq1$, if $\mu$ is a mean zero Gaussian measure on $\mathbb{R}^n$, then for $K,M$ convex closed subsets of $\mathbb{R}^n$ symmetric around the origin, we have,
\begin{eqnarray}
\mu(K\cap M)&\geq&\mu(K)\mu(M)\nonumber
\end{eqnarray}
\end{theorem}
\begin{lemma}\label{Lemma: Gaussian correlation inequality for GP}
Let $(l(t))_{t\geq 0}$ denote a Gaussian Process with  zero mean, stationary covariance function $\kappa$ and continuous sample paths, then
\begin{eqnarray}
\Prob\left(\underset{t\in[0,\tau]}{\sup}|l(t)|\leq K_1,\underset{t\in(\tau,\tau_2]}{\sup}|l(t)|\leq K_2\right)&\geq&\Prob\left(\underset{t\in[0,\tau]}{\sup}|l(t)|\leq K_1\right)\Prob\left(\underset{t\in(\tau,\tau_2]}{\sup}|l(t)|\leq K_2\right)
\end{eqnarray}
for $\tau\leq\tau_2$ and $K_1, K_2$ constants.
\end{lemma}
\begin{remark}
Moreover, the same result holds for $\tau_2$ equal to infinity.
\end{remark}
\begin{proof}

Consider the finite collection of indices  $A_n=\left\{t_k=\frac{k\tau_2}{2^n},k=0,\ldots,2^n\right\}$, then by definition of a Gaussian process, $l(A_n)$ is a zero mean (since $(l(t))_{t\geq 0}$ has zero mean) multivariate Gaussian random variable in $\mathbb{R}^{|A_n|}$.  Moreover, let $t_{k^\star}$ the largest element in $A_n$ such that $t_{k^\star}\leq\tau$, then the sets
\begin{eqnarray}
K&=&\left\{|l(t_0)|\leq K_1,\ldots,|l(t_{k^\star})|\leq K_1,l(t_{k^\star+1})\in\mathbb{R},\ldots,l(\tau_2)\in\mathbb{R}\right\}\nonumber\\
M&=&\left\{l(t_0)\in\mathbb{R},\ldots,l(t_{k^\star})\in \mathbb{R},|l(t_{k^\star+1})|\leq K_2,\ldots,|l(\tau_2)|\leq K_2\right\},
\end{eqnarray}
are convex closed subsets of $\mathbb{R}^{|A_n|}$. Therefore, we can apply the Gaussian correlation inequality, Theorem \ref{GaussianCorrelation}, and have that
\begin{eqnarray}
\Prob\left(K\cap M\right)&\geq&\Prob\left(K\right)\Prob\left(M\right),\nonumber
\end{eqnarray}
which can be rewritten as
{\footnotesize{
\begin{eqnarray}
\Prob\left(\underset{t\in[0,\tau]\cap A_n}{\sup}|l(t)|\leq K_1,\underset{t\in(\tau,\tau_2]\cap A_n}{\sup}|l(t)|\leq K_2\right)\
&\geq&\Prob\left(\underset{t\in[0,\tau]\cap A_n}{\sup}|l(t)|\leq K_1\right)\Prob\left(\underset{t\in(\tau,\tau_2]\cap A_n}{\sup}|l(t)|\leq K_2\right).\label{eq:takelimit}
\end{eqnarray}}}
Moreover, notice that by construction $A_{n}\subset A_{n+1}$ and then
{\footnotesize{
\begin{eqnarray}
\left\{\underset{t\in[0,\tau]\cap A_{n+1}}{\sup}|l(t)|\leq K_1,\underset{t\in(\tau,\tau_2]\cap A_{n+1}}{\sup}|l(t)|\leq K_2\right\}&\subseteq&\left\{\underset{t\in[0,\tau]\cap A_n}{\sup}|l(t)|\leq K_1,\underset{t\in(\tau,\tau_2]\cap A_n}{\sup}|l(t)|\leq K_2\right\},\nonumber
\end{eqnarray}
}}
and then, by a standard properties of measures
{\footnotesize{
\begin{eqnarray}
\underset{n\to\infty}{\lim}\Prob\left(\underset{t\in[0,\tau]\cap A_n}{\sup}|l(t)|\leq K_1,\underset{t\in(\tau,\tau_2]\cap A_n}{\sup}|l(t)|\leq K_2\right)&=&\Prob\left(\underset{n\in\mathbb{N}}{\bigcap}\left\{\underset{t\in[0,\tau]\cap A_n}{\sup}|l(t)|\leq K_1,\underset{t\in(\tau,\tau_2]\cap A_n}{\sup}|l(t)|\leq K_2\right\}\right)\nonumber\\
\text{(by a.s. continuity of $(l(t))_{t\geq 0}$)}&=&\Prob\left(\underset{t\in[0,\tau]}{\sup}|l(t)|\leq K_1,\underset{t\in(\tau,\tau_2]}{\sup}|l(t)|\leq K_2\right).\nonumber
\end{eqnarray}
}}
Then, taking limits in both sides of equation \eqref{eq:takelimit}, we get the desired result. By the same arguments, we have the result when $\tau_2$ is infinity.
\end{proof}

\begin{lemma}\label{Lemma: Sup tail greater than M} Let $(\eta(t))_{t\geq 0}$ be a Gaussian process with continuous paths, zero mean and stationary covariance function $\kappa$, satisfying $(\kappa(0)-\kappa(2^{-n}))^{-1}\geq n^{6}$ (Assumption (A1)).  Then there exits some $\tau^\star>1$, such that for all $\tau>\tau^\star$, $M>0$ and $d\in\mathbb{N}$
\begin{eqnarray}
\Pi\left(\underset{t\geq \tau}{\sup}|\eta(t)h_d(t)|\geq M\right)&\leq&p_{d,M,\kappa}(\tau),\nonumber
\end{eqnarray}
with
\begin{eqnarray}
 p_{d,M,\kappa}(\tau)&=&\sum_{j\geq 0}K^{(1)}_{d,M,\kappa}\exp\{-K^{(2)}_{d,M,\kappa}(\tau+j)^2\},\label{prob pdMkappa}
\end{eqnarray}
where $K^{(1)}_{d,M,\kappa}$ and $K^{(2)}_{d,M,\kappa}$ are positive constants that depend on the dimension $d$, constant $M$ and kernel $\kappa$. Furthermore, $ p_{d,M,\kappa}(\tau)$ goes to zero as $\tau$ goes to infinity.
\end{lemma}
\begin{proof}
Define the sets
\begin{eqnarray}
S&=&\left\{\underset{t\geq \tau}{\sup}|\eta(t)h_d(t)|\geq M\right\}\nonumber\\
S_j&=&\left\{\underset{t\in[\tau+j,\tau+j+1]}{\sup}|\eta(t)|\geq h_d(\tau+j)^{-1}M\right\}\nonumber,
\end{eqnarray}
then, since the function $h_d$ is strictly decreasing for all $t\geq \tau>1$, we have
\begin{eqnarray}
\Pi(S)\leq\Pi\left(\bigcup_{j\geq 0}S_j\right)
\leq\sum_{j\geq 0}\Pi(S_j).\nonumber
\end{eqnarray}
Consider the partition $A_{n}^j=\{ t_{j,n,k}=\tau+j+\frac{k}{2^{n}},k=0,\ldots,2^n\}$. Since the paths of the Gaussian process $(\eta(t))_{t\geq 0}$ are continuous almost surely, by Lemma \ref{Lemma:sup ineq 1} we have
\begin{eqnarray}
\underset{t\in[\tau+j,\tau+j+1]}{\sup}|\eta(t)|\leq |\eta(\tau+j)|+\sum_{n=1}^{\infty}\underset{0\leq k\leq 2^n-1}{\max}|\eta(t_{j,n,k+1})-\eta(t_{j,n,k})|+|\eta(\tau+j+1)|.\nonumber
\end{eqnarray}
Define the set 
\begin{eqnarray}
Z_j&\coloneqq&\left\{|\eta(\tau+j)|\geq \frac{h_d(\tau+j)^{-1}M}{4}\right\}\cup\left\{ \bigcup_{n\geq 1}\left\{\underset{0\leq k\leq 2^{n}-1}{\max}|\eta(t_{j,n,k+1})-\eta(t_{j,n,k})|\geq\frac{3h_d(\tau+j)^{-1}M}{\pi^2n^2}\right\}\right\}\nonumber\\
&\cup&\left\{|\eta(\tau+j+1)|\geq\frac{h_d(\tau+j)^{-1}M}{4}\right\},\nonumber
\end{eqnarray}
then we prove 
\begin{eqnarray}
S_j\subseteq Z_j,\nonumber
\end{eqnarray}
for all $j\in\{0,\ldots,d\}$. We proceed by contradiction. Take a function $\eta$ in $S_j$ and suppose it is not in $Z_j$, then by construction
\begin{eqnarray}
|\eta(\tau+j)|+\sum_{n=1}^{\infty}\underset{0\leq k\leq 2^n-1}{\max}|\eta(t_{j,n,k+1})-\eta(t_{j,n,k})|+|\eta(\tau+j+1)|&<&h_d(\tau+j)^{-1}M,\nonumber
\end{eqnarray}
which contradicts the statement of Lemma 
\ref{Lemma:sup ineq 1}, since for $\eta$ in $S_j$ it holds $\underset{t\in[\tau+j,\tau+j+1]}{\sup}|\eta(t)|\geq h_d(\tau+j)^{-1}M$. Hence,
\begin{eqnarray}
\Pi\left(S_j\right)&\leq&\Pi\left(Z_j\right)\nonumber\\
&\leq&\Pi\left(|\eta(\tau+j)|\geq \frac{h_d(\tau+j)^{-1}M}{4}\right)+\Pi\left(|\eta(\tau+j+1)|\geq\frac{h_d(\tau+j)^{-1}M}{4}\right)\nonumber\\
&+&\sum_{n\geq 1}\sum_{0\leq k\leq 2^{n}-1}\Pi\left(|\eta(t_{j,n,k+1})-\eta(t_{j,n,k})|\geq\frac{3h_d(\tau+j)^{-1}M}{\pi^2n^2}\right).\label{eq: plug-in}
\end{eqnarray}
Recall the Gaussian concentration inequality,
\begin{eqnarray}
\mathbb{P}(|X-\mu|\geq t)\leq 2\exp\left\{-\frac{t^2}{2\sigma^2}\right\},\label{GaussianConcentrationInequality}
\end{eqnarray}
then

\begin{eqnarray}
\Pi\left(|\eta(\tau+j)|\geq \frac{h_d(\tau+j)^{-1}M}{4}\right)\leq 2\exp\left\{\frac{h_d(\tau+j)^{-2}M^2}{32\kappa(0)}\right\}\label{result1}
\end{eqnarray}
and 
\begin{eqnarray}
\Pi\left(|\eta(\tau+j+1)|\geq \frac{h_d(\tau+j)^{-1}M}{4}\right)\leq 2\exp\left\{\frac{h_d(\tau+j)^{-2}M^2}{32\kappa(0)}\right\}.\label{result2}
\end{eqnarray}

Moreover, define the event $I_{j,n,k}=\left\{|\eta(t_{j,n,k+1})-\eta(t_{j,n,k})|\geq\frac{3h_d(\tau+j)^{-1}M}{\pi^2n^2}\right\}$, where the random variable $\eta(t_{j,n,k+1})-\eta(t_{j,n,k+1})$ is Gaussian  with zero mean and variance given by $2(\kappa(0)-\kappa(2^{-n}))$, then using Gaussian concentration inequality, equation \eqref{GaussianConcentrationInequality}
\begin{eqnarray}
\sum_{n\geq 1}\sum_{0\leq k\leq 2^{n}-1}\Pi\left(I_{j,n,k}\right)&\leq&\sum_{n\geq 1}2^{n+1}\exp\left\{-\frac{9h_d(\tau+j)^{-2}M^2}{4\pi^2n^4(\kappa(0)-\kappa(2^{-n}))}\right\}
\nonumber\\
&=&2\sum_{n\geq 1}\exp\left\{-\frac{9h_d(\tau+j)^{-2}M^2(\kappa(0)-\kappa(2^{-n}))^{-1}}{4\pi^4n^4}+n\log 2\right\}.\nonumber
\end{eqnarray}
Since $(\kappa(0)-\kappa(2^{-n}))^{-1}\geq n^6$, assumption (A1)
\begin{eqnarray}
\sum_{n\geq 1}\sum_{0\leq k\leq 2^{n}-1}\Pi\left(I_{j,n,k}\right)&\leq&2\sum_{n\geq 1}\exp\left\{-\frac{9h_d(\tau+j)^{-2}M^2n}{4\pi^4}+n\log 2\right\}\nonumber\\
&=&2\sum_{n\geq 1}\exp\left\{-C_j\right\}^n
,\label{convergent series}
\end{eqnarray}
where $C_j=\frac{9h_d(\tau+j)^{-2}M^2}{4\pi^4}-\log{2}$. Since $h_d(t)$ is strictly decreasing for $t\geq 1$, we can always pick a $\tau^{\star}>1$ in such way that for all $\tau\geq\tau^\star$,
\begin{eqnarray}
\exp\{-C_0\}=\exp\left\{-\frac{9h_d(\tau)^{-2}M^2}{4\pi^4}+\log{2}\right\}<1.
\end{eqnarray}
A direct consequence of the latter result, is the convergence of the geometric series in equation \eqref{convergent series}
\begin{eqnarray}
\eqref{convergent series}&=&2\frac{\exp\{-C_j\}}{1-\exp\{-C_j\}}\leq 2\frac{\exp\{-C_j\}}{1-\exp\{-C_0\}}\leq\tilde{K_0}\exp\{-C_j\},\label{result3}
\end{eqnarray}
for all $j\in\mathbb{N}$ and some constant $\tilde{K_0}$ (which depends on $d$ and $M$). Plugging-in the result of equations \eqref{result1}, \eqref{result2} and \eqref{result3} in equation \eqref{eq: plug-in}, we have
\begin{eqnarray}
\Pi\left(S_j\right)&\leq&4\exp\left\{-\frac{h_d(\tau+j)^{-2}M^2}{32\kappa(0)}\right\}+\tilde{K}_0\exp\{-C_j\}.\nonumber
\end{eqnarray}
Finally,
\begin{eqnarray}
\Pi(S)\leq\sum_{j\geq 0}\Pi\left(S_j\right)\nonumber
\leq\sum_{j\geq 0}\left(4\exp\left\{-\frac{h_d(\tau+j)^{-2}M^2}{32\kappa(0)}\right\}+\tilde{K}_0\exp\left\{-\frac{9h_d(\tau+j)^{-2}M^2}{4\pi^4}-\log{2}\right\}\right),\nonumber
\end{eqnarray}
with
\begin{eqnarray}
h_d(t)^{-2}&=&\frac{(t+\log(1-e^{-t}))^2}{(d+1)^2}.\nonumber
\end{eqnarray}
Therefore there exists some positive constants $K^{(1)}_{d,M,\kappa}$ and $K^{(2)}_{d,M,\kappa}$
\begin{eqnarray}
\Pi(S)&\leq&\sum_{j\geq 0}K^{(1)}_{d,M,\kappa}\exp\{-K^{(2)}_{d,M,\kappa}(\tau+j)^2\}.
\end{eqnarray}
From an standard application of dominated convergence, we deduce that the limit goes to zero as $\tau$ goes to infinity for fixed $M$ and $d$. 
\end{proof}
\begin{proof}[Proof of Lemma \ref{Lemma: bounded supremum}]
By definition, $(\eta_j(t))_{t\geq 0}$ is a Gaussian process with continuous sample paths, zero mean and stationary covariance function $k_j$, for all $j\in\{0,\ldots,d\}$. Since $h_d$ is continuous then $(\hat{\eta}_j(t))_{t\geq 0}$ has also continuous sample paths. We proceed to prove
\begin{eqnarray}
\Pi\left(\lim_{\tau\to\infty}\underset{t\geq\tau}{\sup}|\eta_j(t)h_d(t)|=0\right)&=&1,\nonumber
\end{eqnarray}
for all $j\in\{0,\ldots,d\}$.

Under assumption (A1) and Lemma \ref{Lemma: Sup tail greater than M}, there exits $\tau^\star>1$ such that for all $\tau\geq\tau^\star$,
\begin{eqnarray}
\Pi\left(\underset{t\geq\tau}{\sup}|\eta_j(t)h_d(t)|\geq M\right)&\leq&p_{d,M,\kappa_j}(\tau),\nonumber
\end{eqnarray}
for all $j\in\{0,\ldots,d\}$ (take the maximum $\tau^\star$ as we have a finite collection of Gaussian processes) and with $p_{d,M,\kappa_j}(\tau)$ defined in equation \eqref{prob pdMkappa}. In particular, by Lemma \ref{Lemma: Sup tail greater than M}, this function decreases to zero as $\tau$ goes to infinity for any fixed value of $M>0$ and $d\in\mathbb{N}$. 

Using the latter result, there exist a $\tau^\star>1$ such that for all $\tau>\tau^\star$ 
\begin{eqnarray}
\Pi\left(\underset{t\geq\tau}{\sup}|\eta_j(t)h_d(t)|\geq 1\right)&\leq&p_{d,1,\kappa_j}(\tau),\nonumber
\end{eqnarray}
for all $j\in\{0,\ldots,d\}$. Moreover, consider the increasing sequence of times $\tau_{j,1},\tau_{j,2}\ldots$ in the following way. Take $\tau_{j,1}=\tau>\tau^\star$ and choose $\tau_{j,i}$ such that the probability of the event 
\begin{eqnarray}
E_{j,i}&=&\left\{\underset{t\geq\tau_{j,i}}{\sup}|\eta_j(t)h_d(t)|\geq \frac{1}{i}\right\},\nonumber
\end{eqnarray}
is bounded as
\begin{eqnarray}
\Pi\left(E_{j,i}\right)\leq p_{d,\frac{1}{i},\kappa_j}(\tau_{j,i})\leq \frac{p_{d,1,\kappa_j}(\tau)}{i^2},\nonumber
\end{eqnarray} 
this can be done since the function $p_{d,1,\kappa_j}(\tau)$ decreases to zero in $\tau$. Therefore
\begin{eqnarray}
\sum_{i=1}^\infty\Pi\left(E_{j,i}\right)\leq\sum_{i=1}^{\infty}\frac{p_{d,1,\kappa_j}(\tau)}{i^2}=\frac{p_{d,1,\kappa_j}(\tau)\pi^2}{6}<\infty.\nonumber
\end{eqnarray}
By the Borel Cantelli lemma 
\begin{eqnarray}
\limsup_{i\to\infty}\underset{t\geq\tau_{j,i}}{\sup}|\eta_j(t)h_d(t)|\leq\limsup_{i\to\infty}\frac{1}{i}=0.\nonumber
\end{eqnarray}
$\mathbb{P}$-a.s. for all $j\in\{0,\ldots,d\}$.
\end{proof}

\begin{lemma}\label{Lemma: Positivity [0,tau]} Let $(\eta(t))_{t\geq 0}$ be a Gaussian process with continuous sample paths, zero mean and stationary covariance function $\kappa$, satisfying $(\kappa(0)-\kappa(2^{-n}))^{-1}\geq n^{6}$. Let $\tau\geq1$ and $\delta>0$, then
\begin{eqnarray}
\Pi\left(\underset{t\in[0,\tau]}{\sup}|\eta(t)h_d(t)|\leq \frac{\delta h_d(\tau)}{1+\tau}\right)&>&0.\nonumber
\end{eqnarray} 
\end{lemma}
\begin{proof}
Recall the definition of function $h_d$ given in equation \eqref{eq: h function}
\begin{eqnarray}
h_d(t)&=&\left\{\begin{array}{cc}
\frac{d+1}{1+\log(1-e^{-1})}&t\leq 1\\
\frac{d+1}{t+\log(1-e^{-t})}&t> 1
\end{array}\right. ,\nonumber
\end{eqnarray}
hence $h_d(t)$ is strictly decreasing and positive for $t\geq 1$. Consider $\tau\geq 1$, then 
\begin{eqnarray}
\left\{\underset{t\in[0,\tau]}{\sup}|\eta(t)h_d(t)|\leq \frac{\delta h_d(\tau)}{1+\tau}\right\}&\supseteq&\left\{\underset{t\in[0,\tau]}{\sup}|\eta(t)|\leq  \frac{\delta h_d(\tau)}{h_d(1)(1+\tau)}\right\}.\nonumber
\end{eqnarray}
Consider the partition $A_n=\left\{t^n_k=\frac{k\tau}{2^n},k=0,\ldots,2^n\right\}$, by Lemma \ref{Lemma:sup ineq 1}
\begin{eqnarray}
\underset{t\in[0,\tau]}{\sup}|\eta(t)|&\leq& |\eta(0)|+|\eta(\tau)|+\sum_{n=1}^{\infty}\underset{0\leq k\leq 2^{n}-1}{\max}\left|\eta(t^n_{k+1})-\eta(t^n_k)\right|.\nonumber
\end{eqnarray}
For simplicity, let's call $\psi_\delta(\tau)=\delta\frac{h_d(\tau)}{h_d(1)(1+\tau)}$ and define the event
\begin{eqnarray}
S&=&\left\{\underset{t\in[0,\tau]}{\sup}|\eta(t)|\leq \psi_\delta(\tau)\right\}.\nonumber
\end{eqnarray}
On the other hand, define the event
\begin{eqnarray}
Z&\coloneqq&\left\{|\eta(0)|\leq\frac{\psi_\delta(\tau)}{4}\right\}\cap\left\{\ |\eta(\tau)|\leq\frac{\psi_\delta(\tau)}{4}\right\}\cap\left\{\bigcap_{n=1}^\infty\left\{ \underset{0\leq k\leq 2^n-1}{\max}\left|\eta(t^n_{k+1})-\eta(t^n_k)\right|\leq\frac{3\psi_\delta(\tau)}{n^2\pi^2}\right\}\right\}.\nonumber
\end{eqnarray}
Then, for any function $\eta$ in $Z$, it holds
\begin{eqnarray}
\underset{t\in[0,\tau]}{\sup}|\eta(t)|&\leq& |\eta(0)|+|\eta(\tau)|+\sum_{n=1}^{\infty}\underset{0\leq k\leq 2^{n}-1}{\max}\left|\eta(t^n_{k+1})-\eta(t^n_k)\right|\leq\psi_\delta(\tau)\nonumber,
\end{eqnarray}
which means it also belongs to $S$ and then
\begin{eqnarray}
\Pi(S)&\geq&\Pi(Z).\nonumber
\end{eqnarray}
Moreover, define the event
\begin{eqnarray}
Z_N&\coloneqq&\left\{|\eta(0)|\leq\frac{\psi_\delta(\tau)}{4}\right\}\cap\left\{\ |\eta(\tau)|\leq\frac{\psi_\delta(\tau)}{4}\right\}\cap\left\{\bigcap_{n=1}^N\left\{ \underset{0\leq k\leq 2^n-1}{\max}\left|\eta(t^n_{k+1})-\eta(t^n_k)\right|\leq\frac{3\psi_\delta(\tau)}{n^2\pi^2}\right\}\right\}.\nonumber
\end{eqnarray}
Consider the collection of random variables $\{\alpha^ n_k=\eta(t^n_{k+1})-\eta(t^n_k),k\in\{0,\ldots,2^n-1\},n\in\{1,\ldots,N\}\}$, $\eta(0)$ and $\eta(\tau)$. Observe there are in total $2+\sum_{n=1}^N2^n=2^{N+1}$ variables, then the joint distribution of them is a zero mean, $2^{N+1}$-variate Gaussian distribution, i.e. they lie on $\mathbb{R}^{2^{N+1}}$.

Therefore, by replicating the arguments given in the proof of Lemma \ref{Lemma: Gaussian correlation inequality for GP} (since we have convex, closed and symmetric around the origin events on $\mathbb{R}^{2^{N+1}}$), we use the Gaussian correlation inequality, Theorem \ref{GaussianCorrelation} \cite{memarian2013gaussian}, to prove
\begin{eqnarray}
\Pi(Z_N)&\geq&\Pi\left(|\eta_0|\leq\frac{\psi_\delta(\tau)}{4}\right)\Pi\left(|\eta_\tau|\leq\frac{\psi_\delta(\tau)}{4}\right)\prod_{n=1}^N\prod_{k=0}^{2^n-1}\Pi\left(\left|\eta_{t_{k+1}}-\eta_{t_k}\right|\leq\frac{3\psi_\delta(\tau)}{n^2\pi^2}\right).\nonumber
\end{eqnarray}
Additionally, since the events $Z_N\subseteq Z_{N+1}$ for all $N\geq 1$, 
\begin{eqnarray}
\underset{N\to\infty}{\lim}\Pi\left(\bigcap_{n=1}^N Z_n\right)&=&\Pi(Z),\nonumber
\end{eqnarray}
and hence
\begin{eqnarray}
\Pi(S)\geq\Pi(Z)\geq\Pi\left(|\eta_0|\leq\frac{\psi_\delta(\tau)}{4}\right)\Pi\left(|\eta_\tau|\leq\frac{\psi_\delta(\tau)}{4}\right)\prod_{n=1}^\infty\prod_{k=0}^{2^n-1}\Pi\left(\left|\eta_{t_{k+1}}-\eta_{t_k}\right|\leq\frac{3\psi_\delta(\tau)}{n^2\pi^2}\right).\nonumber
\end{eqnarray}
By using the Gaussian concentration inequality,
\begin{eqnarray}
\Prob(|X-\mu|>t)&\leq&2\exp\left\{-\frac{t^2}{2\sigma^2}\right\},
\end{eqnarray}
(where $X\sim N(\mu,\sigma^2)$), and assumption (A1) $(\kappa(0)-\kappa(2^n))^{-1}\geq n^6$, we have
\begin{eqnarray}
\Pi(S)&\geq&\Pi\left(|\eta_0|\leq\frac{\psi_\delta(\tau)}{4}\right)^2\prod_{n=1}^{\infty}\left(1-\exp\left\{-\frac{9\psi_\delta(\tau)^2(\kappa(0)-\kappa(2^{-n}))^{-1}}{4n^4\pi^2}\right\}\right)^{2^n}\nonumber\\
&\geq&\Pi\left(|\eta_0|\leq\frac{\psi_\delta(\tau)}{4}\right)^2\prod_{n=1}^{\infty}\left(1-\exp\left\{-\frac{9\psi_\delta(\tau)^2n^2}{4\pi^2}\right\}\right)^{2^n}.\nonumber
\end{eqnarray}
Now use $\exp\left\{-\frac{x}{1-x}\right\}\leq 1-x $
\begin{eqnarray}
\Prob(S)&\geq&\Prob\left((|\eta_0|\leq\frac{\psi_\delta(\tau)}{4}\right)^2\prod_{n=1}^{\infty}\exp\left\{-\frac{2^n\exp\left\{-\frac{9\psi_\delta(\tau)^2n^2}{4\pi^2}\right\}}{1-\exp\left\{-\frac{9\psi_\delta(\tau)^2n^2}{4\pi^2}\right\}}\right\}\\
&\geq&\Prob\left((|\eta_0|\leq\frac{\psi_\delta(\tau)}{4}\right)^2\exp\left\{-\sum_{n=1}^{\infty}\frac{\exp\left\{-\frac{9\psi_\delta(\tau)^2n^2}{4\pi^2}+n\log{2}\right\}}{1-\exp\left\{-\frac{9\psi_\delta(\tau)^2n^2}{4\pi^2}\right\}}\right\}>0,\nonumber
\end{eqnarray}
since the latter series converges.

\end{proof}

\begin{proof}[Proof of Lemma \ref{Lemma: Psotivity Joint}]

Consider
\begin{eqnarray}
\Pi(C^j_{\delta,\tau})&=&\Pi\left(\underset{t\in[0,\tau]}{\sup}|\hat{\eta}_j(t)|\leq \frac{\delta h_d(\tau)}{2(1+\tau)},\underset{t> \tau}{\sup}|\hat{\eta}_j(t)|\leq \frac{1}{6}\right),\label{eq1}
\end{eqnarray}
then, by Lemma \ref{Lemma: Gaussian correlation inequality for GP}, we have
\begin{eqnarray}
\eqref{eq1}&\geq&\Pi\left(\underset{t\in[0,\tau]}{\sup}|\hat{\eta_j}(t)|\leq \frac{\delta h_d(\tau)}{2(1+\tau)}\right)\Pi\left(\underset{t\geq \tau}{\sup}|\hat{\eta_j}(t)|\leq \frac{1}{6}\right).\nonumber
\end{eqnarray}
By Lemma \ref{Lemma: Positivity [0,tau]}, for any $\tau>1$ and $\delta>0$, we have
\begin{eqnarray}
\Pi\left(\underset{t\in[0,\tau]}{\sup}|\hat{\eta_j}(t)|\leq \frac{\delta h_d(\tau)}{1+\tau}\right)&>&0,\nonumber
\end{eqnarray}
for all $j\in\{0,\ldots,d\}$. Additionally
\begin{eqnarray}
\Pi\left(\underset{t\geq \tau}{\sup}|\hat{\eta_j}(t)|\leq \frac{1}{6}\right)&=&1-\Pi\left(\underset{t\geq \tau}{\sup}|\hat{\eta_j}(t)|\geq \frac{1}{6}\right),\nonumber
\end{eqnarray}
then, by Lemma \ref{Lemma: Sup tail greater than M}, there exists $\tau^\star_j>1$ such that for $\tau>\tau^\star_j$,
\begin{eqnarray}
\Pi\left(\underset{t\geq \tau}{\sup}|\hat{\eta_j}(t)|\leq \frac{1}{6}\right)&\geq&1-p_{d,\frac{1}{6},\kappa_j}(\tau),\label{Positive set probability}
\end{eqnarray}
where $p_{d,\frac{1}{6},\kappa_j}(\tau)$ decreases to zero for each $j\in\{0,\ldots,d\}$. From the latter, we conclude there exists $\tau_j>\tau^\star_j$ large enough such that for all $\tau\geq\tau_j$
\begin{eqnarray}
\Pi\left(\underset{t\geq \tau}{\sup}|\hat{\eta_j}(t)|\leq \frac{1}{6}\right)>0,\nonumber
\end{eqnarray}
for every $j\in\{0,\ldots,d\}$. By considering $\tau_C=\max\{\tau_0,\ldots,\tau_d\}$ we conclude
\begin{eqnarray}
\Pi(C^j_{\delta,\tau})>0.\nonumber
\end{eqnarray}
for all $\tau\geq\tau_C$ for all $j$.
\end{proof}


\newpage
\begin{thebibliography}{10}

\bibitem{barron1999consistency}
Andrew Barron, Mark~J Schervish, Larry Wasserman, et~al.
\newblock The consistency of posterior distributions in nonparametric problems.
\newblock {\em The Annals of Statistics}, 27(2):536--561, 1999.

\bibitem{choi2004posterior}
Taeryon Choi and Mark~J Schervish.
\newblock Posterior consistency in nonparametric regression problems under
  gaussian process priors.
\newblock 2004.

\bibitem{choudhuri2004bayesian}
Nidhan Choudhuri, Subhashis Ghosal, and Anindya Roy.
\newblock Bayesian estimation of the spectral density of a time series.
\newblock {\em Journal of the American Statistical Association},
  99(468):1050--1059, 2004.

\bibitem{devroye2012combinatorial}
Luc Devroye and G{\'a}bor Lugosi.
\newblock {\em Combinatorial methods in density estimation}.
\newblock Springer Science \& Business Media, 2012.

\bibitem{doksum1974tailfree}
Kjell Doksum.
\newblock Tailfree and neutral random probabilities and their posterior
  distributions.
\newblock {\em The Annals of Probability}, pages 183--201, 1974.

\bibitem{Doob1949}
Joseph~L Doob.
\newblock Application of the theory of martingales.
\newblock {\em Le calcul des probabilites et ses applications}, pages 23--27,
  1949.

\bibitem{dykstra1981bayesian}
RL~Dykstra and Purushottam Laud.
\newblock A bayesian nonparametric approach to reliability.
\newblock {\em The Annals of Statistics}, pages 356--367, 1981.

\bibitem{fernandez}
Tamara Fern\'andez, Nicol\'as Rivera, and Yee~Whye Teh.
\newblock Gaussian processes for survival analysis.
\newblock In {\em Advances in Neural Information Processing Systems}, page to
  appear, 2016.

\bibitem{ghosal2006posterior}
Subhashis Ghosal and Anindya Roy.
\newblock Posterior consistency of gaussian process prior for nonparametric
  binary regression.
\newblock {\em The Annals of Statistics}, pages 2413--2429, 2006.

\bibitem{hjort1990nonparametric}
Nils~Lid Hjort.
\newblock Nonparametric bayes estimators based on beta processes in models for
  life history data.
\newblock {\em The Annals of Statistics}, pages 1259--1294, 1990.

\bibitem{bayes2010}
Nils~Lid Hjort, Chris Holmes, Peter M{\"u}ller, and Stephen~G Walker.
\newblock {\em Bayesian nonparametrics}, volume~28.
\newblock Cambridge University Press, 2010.

\bibitem{kuelbs1994gaussian}
James Kuelbs, Wenbo~V Li, and Werner Linde.
\newblock The gaussian measure of shifted balls.
\newblock {\em Probability Theory and Related Fields}, 98(2):143--162, 1994.

\bibitem{memarian2013gaussian}
Yashar Memarian.
\newblock The gaussian correlation conjecture proof.
\newblock {\em arXiv preprint arXiv:1310.8099}, 2013.

\bibitem{schwartz1965bayes}
Lorraine Schwartz.
\newblock On bayes procedures.
\newblock {\em Zeitschrift f{\"u}r Wahrscheinlichkeitstheorie und verwandte
  Gebiete}, 4(1):10--26, 1965.

\bibitem{tokdar2007posterior}
Surya~T Tokdar and Jayanta~K Ghosh.
\newblock Posterior consistency of logistic gaussian process priors in density
  estimation.
\newblock {\em Journal of Statistical Planning and Inference}, 137(1):34--42,
  2007.

\bibitem{van2008reproducing}
Aad~W van~der Vaart, J~Harry van Zanten, et~al.
\newblock Reproducing kernel hilbert spaces of gaussian priors.
\newblock In {\em Pushing the limits of contemporary statistics: contributions
  in honor of Jayanta K. Ghosh}, pages 200--222. Institute of Mathematical
  Statistics, 2008.

\end{thebibliography}
\end{document}